\documentclass{amsart}
\usepackage[applemac]{inputenc}

\usepackage{color,graphicx}

\usepackage[colorinlistoftodos]{todonotes}

\usepackage{url}
\usepackage[colorlinks=true]{hyperref}
\usepackage{doi}

\newtheorem{theorem}{Theorem}[section]
\newtheorem{lemma}[theorem]{Lemma}
\newtheorem{proposition}[theorem]{Proposition}
\newtheorem{corollary}[theorem]{Corollary}
\newtheorem{conjecture}[theorem]{Conjecture}
\theoremstyle{definition}
\newtheorem{definition}[theorem]{Definition}
\newtheorem{remark}[theorem]{Remark}
\newtheorem{example}[theorem]{Example}

\newcommand{\ocultar}[1]{}

\newcommand{\Q}{\mathbf{Q}}
\newcommand{\cc}{\mathbf{c}}

\newcommand{\N}{\mathbb{N}}
\newcommand{\R}{\mathbb{R}}
\newcommand{\RP}{\mathbb{RP}}

\newcommand{\Z}{\mathbb{Z}}

\newcommand{\proj}{\mathbb{P}}
\newcommand{\bnn}{{\binom{[n]}2}}
\newcommand{\p}{\mathbf{p}}
\newcommand{\RR}{\mathcal{R}}
\newcommand{\CC}{\mathcal{C}}
\newcommand{\VV}{\mathcal{V}}
\newcommand{\HH}{\mathcal{H}}
\newcommand{\PP}{\mathcal{P}}

\newcommand{\q}{\mathbf{q}}
\newcommand{\bt}{\mathbf{t}}

\newcommand{\lin}{\operatorname{lin}}
\newcommand{\link}{\operatorname{lk}}

\newcommand{\Ass}[2]{\Delta_{#1}(#2)}
\newcommand{\OvAss}[2]{\overline{\Delta}_{#1}(#2)}

\DeclareMathOperator{\conv}{conv}
\DeclareMathOperator{\cone}{cone}

\DeclareMathOperator{\sign}{sign}

\title{Realizations of multiassociahedra via rigidity}
\author{Luis Crespo Ruiz \and Francisco Santos}

\thanks{{\bf Funding:} Supported by grants PID2019-106188GB-I00 and PID2022-137283NB-C21 and PRE2020-092702 funded by MCIN/AEI/10.13039/501100011033, and by project CLaPPo (21.SI03.64658) of Universidad de Cantabria and Banco Santander.}

\address{
Departamento de Matem\'aticas, Estad\'istica y Computaci\'on,
Universidad de Canta\-bria,
39005 Santander, Spain
}
\email{francisco.santos@unican.es, luis.cresporuiz@unican.es}

\begin{document}

\begin{abstract}
    Let $\Ass{k}{n}$ denote the simplicial complex of $(k+1)$-crossing-free subsets of edges in $\bnn$. Here $k,n\in \N$ and $n\ge 2k+1$.
    Jonsson (2003) proved that (neglecting the short edges that cannot be part of any $(k+1)$-crossing), $\Ass{k}{n}$ is a shellable sphere of dimension $k(n-2k-1)-1$, and conjectured it to be polytopal. The same result and question arose in the work of Knutson and Miller (2004) on subword complexes.
    
Despite considerable effort, the only values of $(k,n)$ for which the conjecture is known to hold are
$n\le 2k+3$ (Pilaud and Santos, 2012) and $(2,8)$ (Bokowski and Pilaud, 2009).
    Using ideas from rigidity theory and choosing points along the moment curve we realize $\Ass{k}{n}$ as a polytope for $(k,n)\in \{(2,9), (2,10) , (3,10)\}$. We also realize it as a simplicial fan for all $n\le 13$ and arbitrary $k$, except the pairs $(3,12)$ and $(3,13)$. 
    
    Finally, we also show that for $k\ge 3$ and $n\ge 2k+6$ no choice of points can realize $\Ass{k}{n}$  via bar-and-joint rigidity with points along the moment curve or, more generally, via cofactor rigidity with arbitrary points in convex position.
\end{abstract}

\maketitle

\tableofcontents

\section{Introduction and statement of results}
\label{sec:intro}

\subsection*{Introduction}

Triangulations of the convex $n$-gon $P$ ($n > 2$), regarded as sets of edges,  are the facets of an abstract simplicial complex with vertex set $\bnn$ and defined by taking as simplices all the non-crossing sets of diagonals. This simplicial complex, ignoring the boundary edges $\{i,i+1\}$, is a polytopal sphere of dimension $n-4$ dual to the \emph{associahedron}.
(Here and all throughout the paper,  indices for vertices of the $n$-gon are regarded modulo $n$).

A similar complex can be defined if, instead of forbidding pairwise crossings, we forbid crossings of more than a certain number  of edges. More precisely, we say that a subset of $\bnn$ is \emph{$(k+1)$-crossing-free},  (assuming $n \ge 2k+2$) if it does not contain $k+1$ edges that mutually cross, and  
define  $\Ass{k}{n}$ as the simplicial complex consisting of $(k+1)$-crossing-free sets of diagonals. Its facets are called \emph{$k$-triangulations} since in the case $k=1$ they are exactly the triangulations of the $n$-gon.

The $nk$ diagonals of length at most $k$ (where length is measured cyclically) belong to every $k$-triangulation since they cannot participate in any $(k+1)$-crossing. Hence, it makes sense to define the reduced complex $\OvAss{k}{n}$ obtained from $\Ass{k}{n}$ by forgetting them. We call  $\OvAss{k}{n}$ the  \emph{multiassociahedron} or {$k$-associahedron}. 
See Section~\ref{sec:multi} for more precise definitions, and \cite{PilPoc,PilSan,Stump} for additional information.

It was proved in \cite{Naka,DKM} that every $k$-triangulation of the $n$-gon has exactly $k(2n-2k-1)$ diagonals. That is, $\Ass{k}{n}$ is pure of dimension $k(2n-2k-1)-1$.
Jonsson \cite{Jonsson1} further proved that the reduced version $\OvAss{k}{n}$ is a vertex-decomposable (hence shellable) sphere of dimension $k(n-2k-1)-1$, and conjectured it to be polytopal. Remember that all polytopal spheres are shellable, so shellability can be considered evidence in favor of polytopality.  Vertex-decomposability is a stronger notion introduced by Provan and Billera~\cite{ProBil} implying, for example, that the diameters of  these spheres satisfy the Hirsch bound.

\begin{conjecture}[Jonsson]
\label{conj:polytope}
For every $n\ge 2k+1$ the complex $\OvAss{k}{n}$ is a polytopal sphere. That is, there is a simplicial polytope of dimension $k(n-2k-1)-1$ and with $\binom{n}2-kn$ vertices whose lattice of proper faces is isomorphic to $\OvAss{k}{n}$.
\end{conjecture}

The first appearance of this statement, as a question rather than a conjecture, is the 2003 preprint~\cite{Jonsson1}. The conjecture then appeared explicitly in Jonsson's hand-written abstract  after his talk in an Oberwolfach Workshop the same year~\cite{mfo-2003,Jonsson-handwritten} (but it did not appear in the shorter abstract published in the Oberwolfach Reports). It was also included in the unpublished manuscript by Dress, Gr\"unewald, Jonsson, and Moulton~\cite{Jonsson-etal}, before appearing in papers by other authors \cite{PilSan,Stump}.

\begin{remark}
The question of polytopality of $\OvAss{k}{n}$ is quite natural, since it generalizes the associahedron (the case $k=1$) which admits many different constructions as a polytope~\cite{CSZ,PSZ}. One would expect that, as happens in the case of the associahedron, having explicit polytopal constructions of $\OvAss{k}{n}$ would uncover interesting combinatorics. If, in the contrary, it turns out that $\OvAss{k}{n}$ is not always polytopal, it would also be interesting to know it; it would probably be the first family of shellable spheres naturally arising from a combinatorial problem and that are proven not to be polytopal.

Interest in this question comes also from cluster algebras and Coxeter combinatorics. 
 Let $w\in W$ be an element in a Coxeter group $W$ and let $Q$ be a word of a certain length $N$. Assume that $Q$ contains as a subword a reduced expression for $w$. The \emph{subword complex} of $Q$ and $w$ is the simplicial complex with vertex set $[N]$ and with faces the subsets of positions that can be deleted from $Q$ and still contain a reduced expression for $w$. Knutson and Miller~\cite[Theorem 3.7 and Question 6.4]{KnuMil} proved that every subword complex is either a vertex-decomposable (hence shellable) ball or sphere, and they asked whether all spherical subword complexes are polytopal. 
 It was later proved by Stump~\cite[Theorem 2.1]{Stump} that  $\OvAss{k}{n}$ is a spherical subword complex for the Coxeter system of type $A_{n-2k-1}$ and, moreover, it is \emph{universal}: every other spherical subword complex of type $A$ appears as a link in some $\OvAss{k}{n}$~\cite[Proposition 5.6]{PilSan:brick}. 
  In particular, Conjecture~\ref{conj:polytope} is equivalent to a positive answer (in type A) to the question of Knutson and Miller. 
 
 Versions of $k$-associahedra for the rest of finite Coxeter groups exist, with the same implications~\cite{CLS14}.
\end{remark}

Conjecture~\ref{conj:polytope} is easy to prove  for $n\le 2k+3$. $\OvAss{k}{2k+1}$ is indeed a $-1$-sphere (the complex whose only face is the empty set). $\OvAss{k}{2k+2}$ is the face poset of a $(k-1)$-simplex, and $\OvAss{k}{2k+3}$ is (the polar of) the cyclic polytope of dimension $2k-1$ with $n$ vertices (Lemma 8.7 in \cite{PilSan}). 
The only additional case for which Jonsson's conjecture is known to hold is $k=2$ and $n=8$ \cite{bokpil}. In some additional cases $\OvAss{k}{n}$ has been realized as a complete simplicial fan, but it is open whether this fan is polytopal. This includes the cases $n\le 2k+4$ \cite{bcl}, the cases $k=2$ and $n\le 13$ \cite{Manneville} and the cases $k=3$ and $n\le 11$ \cite{bcl}.

In this paper we explore Conjecture~\ref{conj:polytope} both in its polytopality version and in the weaker version where we want to realize $\OvAss{k}{n}$ as a complete fan. Our method is to use as rays for the fan the row vectors of a rigidity matrix of $n$ points in dimension $2k$, which has exactly the required rank $k(2n-2k-1)$ for $\Ass{k}{n}$. 
There are several versions of rigidity that can be used, most notably bar-and-joint, hyperconnectivity, and cofactor rigidity. Among these, cofactor rigidity seems the most natural one because it deals with points in the plane; the ``dimension'' $2k$ of this rigidity theory relates to the degree of the polynomials used.

Our results are of two types. On the one hand we show new cases of multiassociahedra $\OvAss{k}{n}$ that can be realized, be it as fans or as polytopes, with cofactor rigidity taking points along the parabola (which is known to be equivalent to bar-and-joint rigidity with points along the moment curve). On the other hand we show that certain multiassociahedra, namely those with $k\ge 3$ and $n\ge 2k+6$ cannot be realized as fans with cofactor rigidity, no matter how we choose the points.

\subsection*{Summary of methods and results}

Using a (human guided) computer search, we find explicit embeddings  of $\OvAss{k}{n}$ for additional parameters, be it as a polytope or only as a complete fan. 
We list only the ones that were not previously known:

\begin{theorem}
\label{thm:main}
\begin{enumerate}
\item For $(k,n) \in \{(2,9), (2,10), (3,10)\}$, $\OvAss{k}{n}$ is a polytopal sphere. 
\item $\OvAss{4}{13}$ can be realized as a complete simplicial fan.
\end{enumerate}
\end{theorem}

Adding this to previous results, we have that $\OvAss{k}{n}$ can be realized as a fan (which for us always means a complete fan) if $n\le\max\{2k+4,13\}$ except for $(n,k)=(3,12)$ and $(3,13)$,  and as a polytope if $n\le\max\{2k+3,10\}$.

\medskip

Our method to realize $\OvAss{k}{n}$ is via rigidity theory. We now explain the connection.
The number $k(2n-2k-1) = 2kn - \binom{2k+1}{2}$ of edges in a $k$-triangulation of the $n$-gon
happens to coincide with the rank of \emph{abstract rigidity matroids} of dimension $2k$ on $n$ elements,  which capture and generalize the combinatorial rigidity of graphs with $n$ vertices  embedded in $\R^{2k}$. This numerical coincidence (plus some evidence) led \cite{PilSan} to conjecture that \emph{all $k$-triangulations of the $n$-gon are bases in the generic bar-and-joint rigidity matroid of $n$ points in dimension $2k$.}

Apart of its theoretical interest, knowing $k$-triangulations to be bases can be considered a step  towards proving polytopality of $\OvAss{k}{n}$, as follows.
For any given choice of points $p_1,\dots,p_n \in \R^{2k}$ in general position, the rows of their rigidity matrix (see Section~\ref{sec:rigidity}) give a real vector configuration $\VV =\{p_{ij}\}_{i,j}$ of rank ${k(2n-2k-1)}$. The question then is whether using those vectors as generators makes $\OvAss{k}{n}$ be a fan, and whether this fan is polytopal. Being bases is then a partial result: it says that at least the individual cones have the right dimension and are simplicial.

All the realizations of $\OvAss{k}{n}$ that we construct use this strategy for positions  of the points along the \emph{moment curve} $\{(t,t^2,\dots,t^{2k})\in\R^{2k}:t\in \R\}$. The reason to restrict our search to the moment curve is that in our previous paper~\cite{CreSan:moment} we show that, for points along the moment curve, the vector configuration obtained with bar-and-joint rigidity coincides (modulo linear isomorphism) with configurations coming from two other interesting forms of rigidity: Kalai's \emph{hyperconnectivity}~\cite{Kalai} along the moment curve and Billera-Whiteley's \emph{cofactor rigidity}~\cite{Whiteley} along the parabola. This is useful in our proofs and it also makes our realizations more ``natural'', since they can be interpreted in the three versions of rigidity.

In fact, we pose the conjecture that positions along the moment curve realizing $\OvAss{k}{n}$ as a basis collection exist for every $k$ and $n$:

\begin{conjecture}
\label{conj:rigid}
$k$-triangulations of the $n$-gon are isostatic (that is, bases) in the bar-and-joint rigidity matroid of generic points along the moment curve in dimension $2k$.
\end{conjecture}

This conjecture implies the one from \cite{PilSan} mentioned above, but it would imply the same for the generic cofactor rigidity matroid and for the generic hyperconnectivity matroid. (The latter is  known to hold by a previous result of ours \cite[Corollary 2.17]{CreSan:Pfaffians}).
As evidence for the conjecture we prove the case $k=2$:

\begin{theorem} \label{thm:rigid}
$2$-triangulations are isostatic in dimension $4$ for generic positions along the moment curve.
\end{theorem}

In fact, our experiments make us believe that in this statement the  word ``generic'' can be changed to ``arbitrary''. 

\begin{conjecture}
\label{conj:rigid-k=2}
$2$-triangulations of the $n$-gon are isostatic (that is, bases) in the bar-and-joint rigidity matroid of arbitrary (distinct) points along the moment curve in dimension $4$.
\end{conjecture}

This conjecture has an apparently much stronger implication:

\begin{theorem}\label{thm:fan-k=2}
	If Conjecture \ref{conj:rigid-k=2} is true, then all positions along the moment curve realize $\OvAss{2}{n}$ as a fan (hence, Conjecture \ref{conj:polytope} would almost be true for $k=2$).
\end{theorem}

So far we have discussed whether $k$-triangulations are bases in the rigidity \emph{matroid}, but for the polytopality question we are also interested in the \emph{oriented matroid}, which tells us the orientation that each $k$-triangulation has as a basis of the vector configuration. The first thing to notice is that now there is a priori not a unique ``generic'' oriented matroid; different generic choices of points may lead to different orientations of the underlying generic matroid.

Since our points lie in the moment curve, we can refer to each point $(t,\dots, t^{2k})$ via its parameter $t$.
The parameters proving Theorem~\ref{thm:main} are as follows:

\begin{itemize}
    \item For $k=2$, the standard positions ($t_i=i$ for each $i$) realize $\OvAss{2}{n}$ as a polytope if and only if $n\le 9$. For $k=2$ and $n\in\{10,11,12,13\}$ they still realize it as a fan, but not as a polytope. Modifying a bit the positions to $(-2,1,2,3,4,5,6,7,9,20)$ we get a polytopal fan for $\OvAss{2}{10}$ (Lemma~\ref{lemma:standard}).

\item Equispaced positions along a circle, mapped to the moment curve via a birational map, realize $\OvAss{k}{n}$ as a fan for every $(k,n)$ with $2k+2\le n\le 13$ except $(3,12)$ and $(3,13)$, and they realize $\OvAss{3}{10}$ as a polytope (Lemma~\ref{lemma:310}).
\end{itemize}

Our experiments show a difference between the case $k=2$, in which all the positions along the moment curve that we have tried realize $\OvAss{k}{n}$ at least as a fan, and the case $k\ge 3$, in which we show that the standard positions do not realize $\OvAss{k}{2k+3}$ as a fan (realizing $\OvAss{k}{n}$ for $n<2k+3$ is sort of trivial):

\begin{theorem}
		\label{thm:Desargues}
		The graph $K_9 - \{16, 37, 49\}$ is a $3$-triangulation of the $n$-gon, but it is a circuit in the cofactor rigidity matroid $\CC_6(\q)$ if the position $\q$ makes the lines through $16$, $37$ and $49$ concurrent.	
This occurs, for example, if we take points along the parabola with $t_i=i$.
	\end{theorem}

This shows that Conjecture~\ref{conj:rigid-k=2} fails for $k\ge 3$, and we prove that it fails in the worst possible way.
We consider this our second main result, after Theorem~\ref{thm:main}:

\begin{theorem}\label{thm:2k+6}
If $k\ge 3$ and $n\ge 2k+6$ then no choice of points $\q\in \R^2$ in convex position makes the cofactor rigidity $C_{2k}(\q)$ realize the $k$-associahedron $\OvAss{k}{n}$ as a fan.
The same happens for bar-and-joint rigidity and for hyperconnectivity with any choice of points along the moment curve.
\end{theorem}

Let us explain this statement. Cofactor rigidity, introduced by Whiteley following work of Billera on the combinatorics of splines, is related to the existence of $(d-2)$-continuous splines of degree $d-1$, for a certain parameter $d$. For this reason it is usually denoted $C_{d-2}^{d-1}$-rigidity, although we prefer to denote it $C_d$-rigidity since, as said above, it induces an example of abstract rigidity matroid of dimension $d$. Since this form of rigidity is based on choosing positions for $n$ points in the plane, it is  the most natural rigidity theory in the context of $k$-triangulations; for any choice $\q$ of $n$ points in convex position in the plane, we have at the same time a convex $n$-gon on which we can model $k$-triangulations and a $2k$-dimensional rigidity matroid $C_{2k}(\q)$ whose rows we can use as vectors to (try to) realize $\OvAss{k}{n}$ as a fan. For $n=2k+3$ we show that this realization, taking as points the vertices of a regular $n$-gon, always realizes $\OvAss{k}{n}$ as a fan (Corollary \ref{coro:2k+3}), but the above statement says that for $n\ge 2k+6$ (and $k\ge 3$) no points in convex position do. As said above, $C_{2k}(\q)$ with points along a parabola is equivalent to bar-and-joint rigidity and to hyperconnectivity with points along the moment curve in $\R^{2k}$.

\begin{remark}
Theorem~\ref{thm:2k+6} still leaves open the possibility of realizing $\OvAss{k}{n}$ for $n\ge 2k+6 \ge 12$ via bar-and-joint rigidity or via hyperconnectivity, but it would need to be with a choice of points not lying in the moment curve. We have not explored this possibility because we cannot think of  a ``natural'' choice of $n$ points in $\R^{2k}$.

Also, observe that for $k\in\{3,4\}$ this theorem and Theorem~\ref{thm:main} (or, rather, its more precise version Lemma~\ref{lemma:310}) completely settle realizability of $\OvAss{k}{n}$ as a fan via cofactor rigidity: it can be done for $n\le 2k+5$ and it cannot for $n\ge 2k+6$.
\end{remark}

\begin{remark}
$k$-triangulations can be bipartized by taking two copies $i^+$ and $i^-$ of each $i\in [n]$ and turning each edge $(i,j)$ of a $k$-triangulation (with $i<j$) into the edge $(i^+,j^-)$. It turns out that when this is done $k$-triangulations are related to fillings of certain polyominos and to bipartite rigidity theory in dimension $k$ (instead of $2k$); for example, readers can easily convince themselves that bipartized triangulations become spanning trees, which are the bases of $1$-dimensional rigidity. Using these ideas,  the follow-up paper~\cite{Crespo:bipartite} by the first author recovers
most of the results of this paper in the context of bipartite rigidity, including an obstruction similar to that of Theorem~\ref{thm:2k+6}.
\end{remark}

From a computational viewpoint, our methods have three parts (see more details in Section~\ref{subsec:experiments}):
\begin{enumerate}
\item First, for given $k,n$, we enumerate all the $k$-triangulations of the $n$-gon. To do this we have adapted code by Vincent Pilaud which uses the relations between $k$-triangulations and sorting networks~\cite{PilPoc}. Although computationally easy, this is the bottleneck of the process because of the large number of $k$-triangulations. (Jonsson~\cite{Jonsson1} proved that the number of $k$-triangulations of the $n$-gon is a Henkel determinant of Catalan numbers, hence growing as ${C_{n}}^{k}$ times a rational function of degree $2k$ in $n$, where $C_n$ denotes the $n$-th Catalan number). 
In all cases where we have been able to enumerate all $k$-triangulations, we have also been able to decide whether given positions realize the fan and/or the polytope.

\item Then, our code tests whether, for given positions, the rigidity matrix realizes the fan or not. We have always used points along the parabola/moment curve (for which the three rigidity theories are equivalent), but the code would work for arbitrary positions and for the three theories. The running time is essentially linear in the number of $k$-triangulations, with a factor depending on $k$ and $n$ since we are doing linear algebra in $\R^{k(2n-2k-1)}$.

\item Finally, when the second step works,  deciding whether the fan (in the obtained realization) is polytopal is equivalent to feasibility of a linear program with one variable for each ray of the  $\binom{n}{2}$ diagonals of the $n$-gon and $k(n-2k-1)$ constraints for each $k$-triangulation.
\end{enumerate}

To choose the positions for the points we use a bit of trial and error. By default we start with equispaced points along the parabola and along the circle, and when both of them fail we modify the positions.

\subsection*{Structure of the paper}

Sections~\ref{sec:multitriangs},~\ref{sec:polytopal},  and~\ref{sec:rigidity}
deal respectively with properties of $k$-triangulations and of $\OvAss{k}{n}$, with how to prove polytopality of a simplicial complex, and with background on the three forms of rigidity (bar-and-joint, cofactor, and hyperconnectivity). They contain mostly introductory and review material but each of them contains also new results. Among them:

\begin{itemize}
\item Corollary~\ref{coro:cycles}: All $1$-dimensional links in any $\OvAss{k}{n}$ have length at most $5$. (That is: all $2$-faces of dual complexes are at most pentagons).

\item Theorem~\ref{thm:winding} and  Proposition~\ref{prop:shortcycles}: if a pure simplicial complex, embedded in a certain vector configuration $\VV$, satisfies the so-called ``interior cocircuit property'' (see \cite[Chapter 4]{triangbook}), for it to be realized as a fan it is sufficient that every link of dimension one is embedded as a simple cycle (a cycle winding only once). This condition holds automatically for cycles of length at most four, so in the case of $\OvAss{k}{n}$ it needs to be checked only for cycles of length five.
\end{itemize}

Section~\ref{sec:obstructions} describes the obstructions for realizability that we have found for $k\ge 3$.  We state and prove them in the context of cofactor rigidity. 
They follow from the so-called \emph{Morgan-Scott split} in the theory of splines, which says that the graph of an octahedron ($K_6$ minus three disjoint edges) is dependent in cofactor $3$-rigidity if and only if the points are in \emph{Desargues position} (that is, the three missing edges are concurrent). After recalling some facts on cofactor rigidity in Section \ref{sec:cofactor}, in Section~\ref{sec:desargues} we prove an oriented version of this obstruction (Theorem \ref{thm:cofact}), and then use it in Section~\ref{sec:obstruction} to prove Theorem~\ref{thm:2k+6}.
Along the way we characterize exactly what choices of points in convex position realize the fan $\OvAss{k}{2k+3}$ (Theorem \ref{thm:star}).

Our final Section~\ref{sec:triangsrigid} contains most of our positive results on realizability of $\OvAss{k}{n}$ with points along the moment curve/parabola.
In Section~\ref{subsec:vertex-split} we prove Theorem~\ref{thm:rigid} along the same lines used in~\cite{PilSan} for bar-and-joint rigidity. In Section~\ref{subsec:cycles} we prove a local version of realizability as a fan: we show that for every $1$-dimensional link in $\OvAss{2}{n}$ there are positions along the moment curve realizing that cycle as a collection of bases satisfying ICoP and with winding number equal to one (Theorem~\ref{thm:good_cycle}), which imply Theorem \ref{thm:fan-k=2}. We also show that  \emph{arbitrary} positions along the moment curve realize $\OvAss{k}{2k+2}$ and $\OvAss{2}{7}$ as polytopes (Corollaries~\ref{coro:2k+2} and~\ref{coro:small-n}). Section~\ref{subsec:experiments} contains details of our experiments, which imply Theorem~\ref{thm:main}.

\subsubsection*{Acknowledgement.} 
We thank Vincent Pilaud for providing us with his code for enumerating $k$-triangulations,
which we then adapted to test polytopality.

\section{Preliminaries and background}

\subsection{Multitriangulations}
\label{sec:multitriangs}
\label{sec:multi}

Let us recall in detail the definition of the $k$-associahedron. As mentioned in the introduction, it is a simplicial complex with vertex set 
\[
\bnn:=\{\{i,j\}: i,j \in [n], i<j\}.
\]

\begin{definition}
Two disjoint elements $\{i,j\},\{k,l\}\in \bnn$, with $i<j$ and $k<l$, of $\binom{[n]}2$ \emph{cross} if $i<k<j<l$ or $k<i<l<j$. That is, if they cross when seen as diagonals of a cyclically labeled convex $n$-gon.

A \emph{$k$-crossing} is a subset of $k$ elements of $\binom{[n]}2$ such that every pair  cross. A subset of $\binom{[n]}2$ is \emph{$(k+1)$-crossing-free} if it doesn't contain any $(k+1)$-crossing. A \emph{$k$-triangulation} is a maximal $(k+1)$-crossing-free set.
\end{definition}

Observe that whether two pairs $\{i,j\},\{k,l\}\in \bnn$ cross is a purely combinatorial concept, but it captures the idea that the corresponding diagonals of a convex $n$-gon  geometrically cross.

The \emph{length} of an edge $\{i,j\}\in\bnn$, is $\min\{|j-i|, n- |j-i|\}$. That is, the distance from $i$ to $j$ measured cyclically in $[n]$.
Edges of length at most $k$ cannot participate in any $k+1$-crossing and, hence, all of them lie in every $k$-triangulation. We call edges of length at least $k+1$ \emph{relevant} and those of length at most $k-1$ \emph{irrelevant}. The ``almost relevant'' edges, those of length $k$, are called \emph{boundary edges} and, although they lie in all $k$-triangulations, they still play an important role in the theory (see Proposition~\ref{prop:stars}(3)).

By definition, $(k+1)$-crossing-free subsets form an abstract simplicial complex on the vertex set $\binom{[n]}2$, whose facets are the $k$-triangulations and whose minimal non-faces are the $(k+1)$-crossings. We denote this complex $\Ass{k}{n}$. Since the $kn$ irrelevant and boundary edges lie in every facet, it makes sense to consider also the \emph{reduced complex} $\OvAss{k}{n}$. Technically speaking, we have that $\Ass{k}{n}$ is the join of $\OvAss{k}{n}$ with the irrelevant face (the face consisting of irrelevant and boundary edges).

Multitriangulations were studied (under a different name) by Capoyleas and Pach~\cite{cp-92}, who showed that no $(k+1)$-crossing-free subset has more than $k(2n-2k-1)$ edges. That is, the complex $\Ass{k}{n}$ has dimension $k(2n-2k-1)-1$, hence $\OvAss{k}{n}$ has dimension $k(n-2k-1)-1$.
The main result about $\OvAss{k}{n}$ for the purposes of this paper is the following particular case of a theorem of Knutson and Miller:

\begin{theorem}[Knutson-Miller \cite{KnuMil}]
\label{thm:sphere}
$\OvAss{k}{n}$ is a shellable sphere of dimension $k(n-2k-1)-1$.
\end{theorem}

The following lemma shows that the realizability question we want to look at is \emph{monotone}; if we have a realization of $\OvAss{k}{n}$ then we also have it for all $\OvAss{k'}{n'}$ with $k'\le k$ and $n'-2k'\le n-2k$. Remember that the link of a face $F$ in a simplicial complex $\Delta$ is
\[
	\link_\Delta(F) := \{ G\in \Delta: G\cap F=\emptyset, G\cup F\in \Delta\} = \{\sigma\setminus F: \sigma\in \Delta, F\subset\sigma\}.
\]
In a shellable $d$-sphere the link of any face of dimension $d'$ is a shellable $d-d'-1$-sphere.

\begin{lemma}[Monotonicity]
\label{lemma:monotone}
Let $n \ge 2k+1$. 
Then, both $\OvAss{k}{n}$ and $\OvAss{k-1}{n-1}$ appear as links in $ \OvAss{k}{n+1}$. More precisely:
\begin{enumerate}
\item 
$\OvAss{k}{n} = \link_{\OvAss{k}{n+1}}(B_{n+1})$, 
where $B_{n+1}:= \{\{n-k+i,i\}: i=1,\dots,k\}\in \bnn$ is the set of edges of length $k+1$ leaving $n+1$ in their short side.
\item
$ \OvAss{k-1}{n-1} \cong \link_{\OvAss{k}{n+1}}(E_{n+1})$,
where $E_{n+1}=\{\{i,n+1\}: i \in [k+1,n-k]\}\in \binom{[n+1]}{2}$ is the set of relevant edges using $n+1$.
\end{enumerate}
\end{lemma}

\begin{proof}
By Theorem~\ref{thm:sphere}, the three complexes $\OvAss{k}{n}$, $ \OvAss{k-1}{n-1}$ and $ \OvAss{k}{n+1}$ are spheres, of the appropriate dimensions. For example,
\[
\dim(\OvAss{k}{n+1}) = k(n+1-2k-1)-1 = k(n-2k)-1.
\]
Since the link of a face of size $j$ in a shellable sphere is a sub-sphere of codimension $j$, the right-hand sides in both equalities are spheres of respective dimensions
\[
\dim(\OvAss{k}{n+1}) - k = k(n-2k) - 1 -k= k(n-2k-1)-1= \dim(\OvAss{k}{n})
\]
in part (1) and
\begin{align*}
\dim(\OvAss{k}{n+1}) - (n-2k) &= k(n-2k) - 1 -(n-2k)\\
&= (k-1)(n-2k)-1 &= \dim(\OvAss{k-1}{n-1})
\end{align*}
in part (2).
Two simplicial spheres of the same dimension cannot be properly contained in one another, so in both equalities we only need to check one containment (with a relabelling of the complex allowed in part (2)) and the other containment then follows automatically.
 
In part (1) we show $\OvAss{k}{n} \subset \link_{\OvAss{k}{n+1}}(B_{n+1})$. That is, for every $k$-triangulation $T$ of the $n$-gon with vertices $[n]$ we need to check that $T\cup B_{n+1}$ is $(k+1)$-crossing-free. This is because all the edges in $B_{n+1}$ have length $k+1$ and have $n+1$ in their short side, so any $(k+1)$-crossing involving one of them needs to use the vertex $n+1$. But $T\cup B_{n+1}$ has no (relevant) edge using $n+1$.

For part (2), we consider the following map
\begin{align*}
\phi: &\bnn \to \binom{[n-1]}{2}\\
& \{i,j\} \mapsto  \{i,j-1\}, \quad 1\le i<j\le n.
\end{align*}
The map $\phi$ is a bijection between the $k$-relevant edges in $\binom{[n+1]}{2}$ not using $n+1$ and the $(k-1)$-relevant edges in $\binom{[n-1]}{2}$. Moreover, the map reduces crossing of pairs of edges. More precisely, $\phi(e)$ and $\phi(f)$ cross if and only if $e$ and $f$ crossed and were not of the form $\{i,j+1\}$, $\{j,\ell+1\}$ for some $1\le i < j < \ell \le n$.

Hence, if $T$ is a $k$-triangulation in $\OvAss{k}{n+1}$ containing $E_{n+1}$ then its image $\phi(T\setminus E_{n+1})$ is $(k+1)$-crossing-free. We need to check that it is also $k$-crossing-free. For this, consider a $(k+1)$-crossing $C$ in $T\setminus E_{n+1}$. Two things can happen:
\begin{itemize}
\item $C$ uses an edge of $E_{n+1}$, so it is no longer a $(k+1)$-crossing in $\phi(T\setminus E_{n+1})$.
\item $C$ does not use any edge of $E_{n+1}$. Then, $C$ is of the form $\{\{a_i,b_i\}:i\in [k+1]\}$ with $1\le a_1 < \dots < a_{k+1} < b_1 < \dots < b_{k+1} \le n$. But we need $b_1 = a_{k+1}+1$, or otherwise  $C\cup \{\{a_{k+1}+1,n+1\}\}$ is a $(k+2)$-crossing in $T\setminus E_{n+1}$. Hence, $\phi(C)$ is no longer a $(k+1)$-crossing because $\phi(\{a_1,b_1\})$ and $\phi(\{a_{k+1},b_{k+1}\})$ do not cross.
\end{itemize}
\end{proof}

%
%

Being a sphere (more precisely, being a pseudo-manifold) has the following important consequence:

\begin{proposition}[Flips \cite{DKM,Naka}, see also \protect{\cite[Lemma 5.1]{PilSan}}]
\label{prop:flips}
For every relevant edge $f$ of a $k$-triangulation $T$ there is a unique edge $e\in\binom{[n]}2$ such that 
\[
T\triangle\{e,f\} := T\setminus \{f\}\cup\{e\}
\]
is another $k$-triangulation.
\end{proposition}

We call the operation that goes from $T$ to $T\triangle\{e,f\}$ a \emph{flip}. The paper~\cite{PilSan} gives a quite explicit description of flips using for this the so-called \emph{stars}:

\begin{definition}[Stars]
\label{defi:stars}
Let $s_0,s_1, \ldots, s_{2k}\in [n]$ be distinct vertices, ordered cyclically. The $k$-star $S$ with this set of vertices is the cycle 
$\{\{s_i,s_{i+k}\}:0\le i\le 2k\}$, with indices taken modulo $2k+1$. 
\end{definition}

In classical terms, a $k$-star is sometimes called a ``star polygon of type $\{2k+1/k\}$'' \cite{Cromwell,GS}.
Observe that every $k$-star $S$ is $(k+1)$-crossing-free but the set $S\cup\{t\}$ where $t$ is a bisector of $S$ is never $(k+1)$-crossing-free. Here, by \emph{bisector} we mean the following:

\begin{definition}[Bisectors]
\label{defi:bisector}
An \emph{angle} consists of two elements $\{a,b\}$ and $\{b,c\}$ in $\bnn$ with a common end-point $b$.  A \emph{bisector} of the angle is any edge $\{b,d\}$ with $d$ lying betwen $a$ and $c$ as seen cyclically from $b$. A \emph{bisector of a star} is a bisector of any of its $2k+1$ angles. That is, 
an edge of the form $\{s_i, t\}$ such that $t$ lies between $s_{i-k}$ and $s_{i+k}$ for some $s_i$ in the star (with the notation of Definition~\ref{defi:stars}).
\end{definition}

In terms of stars and their bisectors, flips can be described as follows:

\begin{proposition}
\label{prop:stars}
Let $T$ be a $k$-triangulation of the $n$-gon. Then:
\begin{enumerate}
    \item $T$ contains exactly $n-2k$ $k$-stars~\cite[Corollary 4.4 and Theorem 3.5]{PilSan}.
    \item Each pair of $k$-stars in $T$ have a unique common bisector~\cite[Theorem 3.5]{PilSan}.
    \item Every relevant edge $e$ in $T$ belongs to exactly two such $k$-stars, and every boundary edge belongs to exactly one~\cite[Corollary 4.2]{PilSan}.
    \item The $k$-triangulation obtained by flipping $e$ in $T$ is $T\triangle\{e,f\}$ where $f$ is the common bisector of the two $k$-stars containing $e$~\cite[Lemma 5.1]{PilSan}.
\end{enumerate}
\end{proposition}


In our next result we ask the following question. Suppose that $F$ is a face in $\OvAss{k}{n}$. That is, $F$ is contained in some $k$-triangulation $T$. How big can the link of $F$ be? 
By ``how big'' we here mean how many vertices (of $\OvAss{k}{n}$, that is, diagonals of the $n$-gon) are used in the link.

\begin{lemma}
\label{lemma:link}
Let $T \in \OvAss{k}{n}$ be a $k$-triangulation and $F\in T$ an edge of it. Let $S$ be the set of $k$-stars of $T$ containing a diagonal in $T\setminus F$. Then, all diagonals used in 
$\link_{\OvAss{k}{n}}(F)$ are either from $T\setminus F$ or common bisectors of two of the $k$-stars in $S$.
\end{lemma}

\begin{proof}
Facets in the link correspond to $k$-triangulations containing $F$. As the flip graph is connected \cite[Corollary 5.4]{PilSan} they can all be obtained from $T$ by iteratively flipping diagonals not belonging to $F$. The diagonals so obtained are bisectors perhaps not of the original stars in $T$ but at least of new stars obtained along the way. However, we can give the following characterization of them: At each vertex $i$ of the $n$-gon, consider (locally) the union of the angles of stars in $S$, seeing each angle as a sector of a small disk centered at $i$. The bisectors of such a union with an endpoint at $i$ are either bisectors (at $i$) of one of the stars, or common edges of two stars. Now, this ``union of stars'' is unchanged by flipping, because each flip removes two stars and inserts another two but with the same union.

Thus, which possible diagonals are used  can be prescribed by looking only at $T$. They either are bisectors of pairs of stars in $T$ or common edges of pairs of stars in $T$. Among the latter we are only allowed to flip, or insert, those that are not in $F$.
\end{proof}

\begin{corollary}
\label{coro:cycles}
All links of dimension one in $\OvAss{k}{n}$ are cycles of length at most five.
\end{corollary}

\begin{proof}
Every such link is a sphere of dimension one, hence a cycle.

Let $F$ be the face we are looking at, so that there are relevant diagonals $\{e,f\}\in \bnn\setminus F$ and a $k$-triangulation $T$ with 
$F=T\setminus\{e,f\}$.
The set $S$ of stars in the previous lemma has size at most four (two for $e$ and two for $f$) but it may have size three or two if $e$ and $f$ belong to one or two common stars. By the lemma, if $|S|$ is two or three then the length of the cycle is (at most) three or five, respectively. 

If $|S|=4$, then each of the flips leaves the two stars corresponding to the other flip untouched. Hence, the two flips commute and the cycle is a quadrilateral, consisting of $T$, its two neighbors by the flip at $e$ or $f$, and the $k$-triangulation obtained by performing both flips, in any order.
\end{proof}

\begin{example}
	The three graphs below represent codimension-two faces of $\OvAss{2}{7}$. 
	In each drawing, blue edges are relevant and two more relevant edges are needed to form a $2$-triangulation.
	The link of the first face is
	a cycle of length 3, consisting of the diagonals $\{15,26,47\}$. In the second, 
	the length is 4: $\{15,25,36,47\}$. In the third
	the length is 5: $\{25,26,36,37,47\}$. 
	(In each case, adding to the given graph any two consecutive edges from the list we get a $2$-triangulation, and all $2$-triangulations containing that graph have this form.)\\
	\centerline{\includegraphics[scale=0.6]{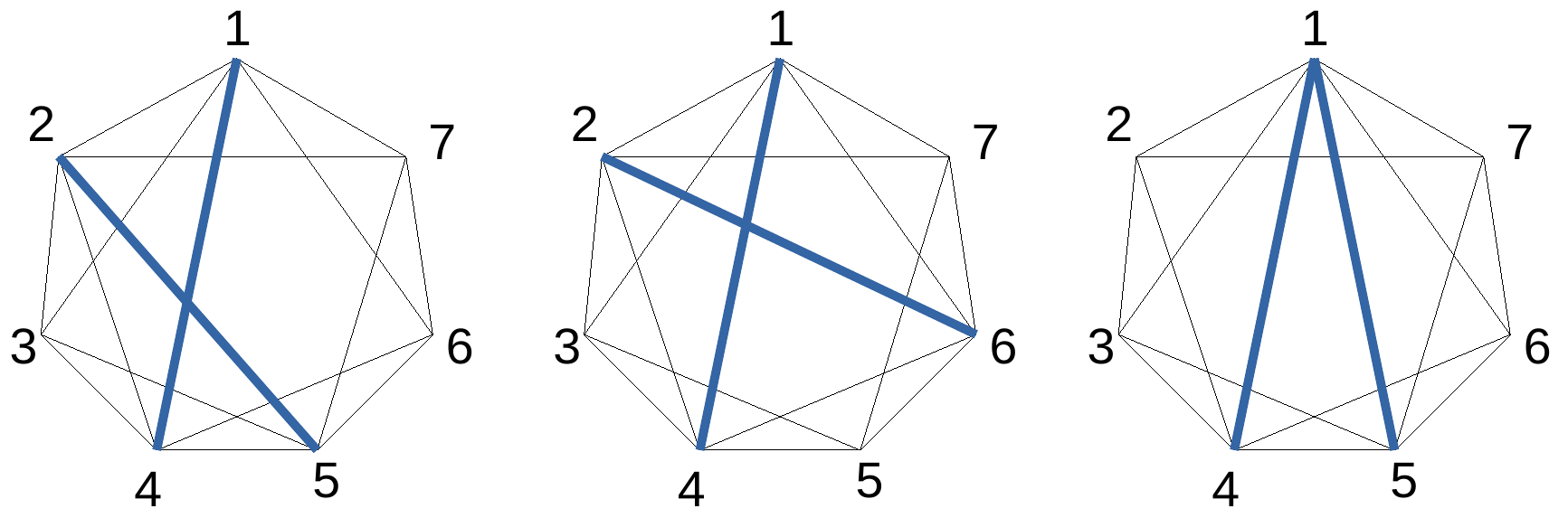}}
	\label{ex:flips}
\end{example}

%
%

\subsection{Polytopality}
\label{sec:polytopal}

Throughout this section $\Delta$ will denote a pure simplicial sphere of dimension $D-1$ with vertex set $V$. We ask ourselves whether $\Delta$ can be realized as the normal fan of a polytope. That is, we ask whether there is a vector configuration $\VV=\{v_i: i\in V\}\subset \R^D$ with the  property that the family of cones 
\[
\{\cone(\VV_F): F\in \Delta\}
\]
form a complete simplicial fan and whether this fan is the normal fan of a polytope. Here we denote
\[
\VV_X:=\{ v_i : i \in X\}
\]
for each $X\subset V$  and 
\[
\cone(X):=\{\lambda_1x_1+\dots+\lambda_s x_s: s\in \N, \lambda_i \in [0,\infty),x_i\in \VV_X\}
\]
is the \emph{cone} generated by $X$. 

The first obvious necessary condition is that we need $\VV_F$ to be linearly independent for each $F\in \Delta$. 
When this happens the cones $\cone(\VV_F)$ are \emph{simplicial cones} and $\VV$ naturally defines a continuous map $\phi_{\Delta,\VV}: |\Delta| \to S^{D-1}$ where $|\Delta|$ denotes the topological realization of $\Delta$ and $S^{D-1}\subset \R^D$ is the unit sphere. The precise definition of $\phi_{\Delta,\VV}$ is as follows: First map $|\Delta|$ to $\R^D$ by sending each vertex $i$ to its corresponding vector $v_i$, and extend this map to $|\Delta|$ linearly within each $F\in \Delta$. The fact that each $F$ is linearly independent ensures that the image of this map does not contain the origin, so we can compose the map with the normalization map $\R^D \to S^{D-1}$ which divides each vector by its $L_2$-norm. 
We call this map $\phi_{\Delta,\VV}$ a \emph{pre-embedding} of $\Delta$. Slightly abusing notation we will also say that \emph{$\VV$ is a pre-embedding of $\Delta$}.

If the pre-embedding happens to be injective, then it is a continuous injective map between two spheres of the same dimension, hence bijective, hence a homeomorphism. This implies that $\Delta$ is a \emph{triangulation of $\VV$} in the sense of \cite{triangbook} (see, e.g., Theorem 4.5.20 in that book). We can also say that, in this case, $\VV$ is an \emph{embedding} of $\Delta$ or that it \emph{realizes $\Delta$ as a complete fan}.

\subsubsection{Conditions for a complete fan}

Remember that the \emph{contraction} of a vector configuration $\VV\subset \R^D$ at an independent subset $I$ is the image of $\VV\setminus I$ under the quotient linear map $\R^D \to \R^D/ \lin(I) \cong \R^{D-|I|}$; it is denoted $\VV/I$.
If $\VV$ is a pre-embedding (resp. an embedding) of $\Delta$ then $\VV/F$ is also a pre-embedding (resp. an embedding) of $\link_\Delta(F)$  for every $F\in \Delta$. We can then consider a hierarchy of embedding properties by asking $\VV/F$ to be an embedding only for faces of at least a certain dimension. The case where $F$ is a facet is trivial. The next level in the hierarchy, when $F$ is a ridge, has received some attention in \cite{triangbook}:

\begin{definition}[ICoP property]
Let $\VV\subset \R^D$ be a pre-embedding of a pure $(D-1)$-complex $\Delta$. We say that the pre-embedding has the the \emph{intersection cocircuit property (ICoP)} if the pre-embedding $\VV_\tau$ of $\link_\Delta(\tau)$ is an embedding for every ridge $\tau$. That is to say, if the following two properties hold:

\begin{itemize}
\item $\tau$ is contained in exactly two facets $\sigma_1$ and $\sigma_2$.

\item The cones $\VV_{\sigma_1}$ and $\VV_{\sigma_2}$ lie in opposite sides of the hyperplane spanned by $\VV_{\tau}$ (which exists and is unique since $\tau$ is independent of size $D-1$). This is equivalent to saying that the unique (modulo a scalar multiple) linear dependence in  $\VV_{\sigma_1\cup \sigma_2}$ has coefficients of the same sign in the two vectors indexed by $\sigma_1\setminus \tau$ and $\sigma_2\setminus \tau$.
\end{itemize}

\end{definition}

Observe that the first condition is independent of $\VV$. When it holds, $\Delta$ is said to be a \emph{pseudo-manifold}.
The pseudo-manifold is \emph{strongly connected} if its dual graph is connected.\footnote{Sometimes strong-connectedness is considered part of the definition of pseudo-manifold, but we do not take this approach.}

Every link in a pseudo-manifold is itself a pseudo-manifold. For example, the link of a codimension-two face $\rho$ is a disjoint union of cycles. We say that $\rho$ is \emph{nonsingular} if $\link_{\Delta}(\rho)$ is a single cycle and in this case we call this cycle the \emph{elementary cycle with center at $\rho$}.

When this happens for every $\rho$ we say that the pseudo-manifold $\Delta$ \emph{has no singularities of codimension two}.
Being a pseudo-manifold with no singularities of codimension two is computationally easy to check: In a pure complex the link of every codimension-two face is a graph, and we only need to check that each of these graphs is a cycle. All manifolds, hence all spheres, hence $\OvAss{k}{n}$ have these properties.

As seen in \cite[Theorem 4.5.20]{triangbook}, the (ICoP) property is almost sufficient for $\Delta$ to be a triangulation of $\VV$, but something else is needed. 
We here express this  ``something else'' in topological terms, in two ways.

Our first characterization is in terms of links of codimension-two faces.
Suppose that $\Delta$ has no singularities of codimension two, so that every face $\rho$ of codimension two defines an elementary cycle. Then, we have that $\VV/\rho$ embeds  $\link_\Delta(\rho)$ as a cyclic collection of cones in $\R^2$, for which we can define its \emph{winding number}: the number of times the cycle wraps around $\R^2\setminus \{0\}$. Homologically, this number is the image in $H_1(\R^2\setminus \{0\},\Z)\cong \Z$ of the elementary cycle as a generator of its homology group $H_1(\link_\Delta(\rho),\Z)\cong \Z$. 
We say that an elementary cycle is \emph{simple} in $\VV$ if its winding number is $\pm1$.

The second characterization is in terms of the degree of the pre-embedding, which is a generalization of winding number to higher dimensions. 
The \emph{degree} of a continuous map $\phi:|\Delta|\to S^{D-1}$ from an orientable $(D-1)$-dimensional pseudo-manifold $\Delta$ to the sphere $S^{D-1}$ can be defined as the image of the fundamental cycle of $H_{D-1}(M, \Z)\cong \Z$ in $H_{D-1}(S^{D-1}, \Z)\cong \Z$. If $\phi$ is injective in each facet
(for example, if it is a pre-embedding as defined above), the degree of $\phi$ can be computed as the number (with sign) of preimages in $\phi^{-1}(y)$ for a sufficiently generic point $y\in S^{D-1}$; ``with sign'' means that each preimage $x\in \phi^{-1}(y)$ counts as $+1$ or $-1$ depending on whether $\phi$ preserves or reverses orientation in the facet containing $x$.

Observe that being a pre-embedding with the (ICoP) property implies $\Delta$ to be an orientable pseudo-manifold.

\begin{theorem}
\label{thm:winding}
\label{thm:signcond}
Let $\VV\subset \R^D$ be a pre-embedding of $\Delta$ with the (ICoP) property. Let $\phi_{\Delta,\VV}: |\Delta| \to S^{D-1}$ be the associated map. Then, the following conditions are equivalent:
\begin{enumerate}
\item $\phi_{\Delta,\VV}$ is a homeomorphism; that is, $\VV$ is a triangulation of $\Delta$ or, equivalently, $\VV$ embeds $\Delta$ as a complete simplicial fan in $\RR^D$.
\item Every sufficiently generic vector $v\in \R^D$  is contained in only one of the facet cones $\{\cone(\VV_\sigma) : \sigma\in \operatorname{facets(\Delta)}\}$.
\item There is some vector $v\in \R^D$ that is contained in only one of the facet cones $\{\cone(\VV_\sigma) : \sigma\in \operatorname{facets(\Delta)}\}$.
\item  $\Delta$ has no singularities of codimension two and all its elementary cycles are simple in $\VV$.
\item $\phi_{\Delta,\VV}$ has degree $\pm 1$.
\end{enumerate}
\end{theorem}

\begin{proof}
We only need to show that any of (4) and (5) implies one of (1), (2) or (3), since
the implications (1)$\Rightarrow$(2)$\Rightarrow$(3), (1)$\Rightarrow$(4) and  (1)$\Rightarrow$(5) are obvious and
 (3)$\Rightarrow$(1) is part of \cite[Corollary 4.5.20]{triangbook}.

Let us see the implication (5)$\Rightarrow$(3).
The property (ICoP) implies that the map $\phi_{\Delta,\VV}$ is consistent with orientations: either it preserves orientations of all facets or reverses orientation of all facets. This implies that when we compute the degree via a generic fiber there are no cancellations and, since the map has degree one, every generic fiber has a single point. That is, $\phi_{\Delta,\VV}$ is injective except perhaps in a subset of measure zero (the $(D-2)$-skeleton of $|\Delta|$), so (3) holds.

For the implication (4)$\Rightarrow$(1) we use induction on $D$.
If $D\le 2$ there is nothing to prove, so we assume $D\ge 3$. 
Since elementary cycles are preserved by taking links/contractions, the inductive hypothesis implies that $\link_\Delta(i)$ is a triangulation of $\VV/v_i$ for every $i\in V$. 
In particular, $\link_\Delta(i)$ is topologically a sphere and, hence, $\Delta$ is a manifold. 
Moreover, again by the inductive hypothesis, the map $\phi_{\Delta,\VV}: |\Delta|\to S^{D-1}$ is a local homeomorphism. 
Every local homeomorphism between two manifolds is a covering map. Now, $D-1\ge 2$ implies that $S^{D-1}$ is simply connected and, since $|\Delta|$ is connected, the  covering map $\phi_{\Delta,\VV}$ is a global homeomorphism. 
\end{proof}

Now, by Corollary~\ref{coro:cycles},  elementary cycles in $\OvAss{k}{n}$ have length bounded by five. 
This  suggests we study Theorem~\ref{thm:winding} in more detail for such cycles:

\begin{proposition}
\label{prop:shortcycles}
Let $\VV$ be a pre-embedding of $\Delta$ with the (ICoP) property. 
Then:
\begin{enumerate}
\item All cycles of length $\le 4$ are automatically simple.
\item Let $\rho$ be a codimension-two face whose elementary cycle $Z$ has length five. Let $i_1,\dots,i_5\in V$ be the vertices of $Z$, in their cyclic order. Then, $Z$ is simple if and only if there are three 
consecutive elements $i_1, i_2, i_3\in Z$ such that the unique linear dependence among the vectors $\{v_{i}: i\in \rho\cup\{i_1, i_2, i_3\}\}$ has opposite sign in $i_2$ than the sign it takes in $i_1$ and $i_3$.
\end{enumerate}
\end{proposition}

\begin{proof}
Let us first explain the condition in part two. 
The (ICoP) property implies that for every three consecutive vertices $i_1, i_2, i_3$ in the elementary cycle (of arbitrary length) of a codimension-two face $\rho$ we have that $i_1$ and $i_3$ lie in opposite sides of the hyperplane spanned by $\rho\cup \{i_2\}$. 
By elementary linear algebra (or oriented matroid theory), this translates to the fact that the unique dependence contained in $\rho\cup\{i_1,i_2,i_3\}$ has the same sign in $i_1$ and $i_3$. 
Similarly, whether this sign equals the one at $i_2$ or not expresses whether $i_3$ lies on the same or different side of $\rho\cup\{i_1\}$ as $i_2$. 
Put differently, it tells us whether the dihedral angles of $i_1i_2$ and $i_2i_3$, as seen from $\rho$ add up to more or less than $\pi$.
(If the dependence vanishes at $i_2$ then the angle is exactly $\pi$).

In general, if $Z=i_1i_2\dots i_ni_1$ is a cycle with center $\rho$ and (after contraction of the vector configuration at $\rho$) it is embedded in $\R^2$ with vectors $w_1,\dots,w_n\in \R^2\setminus \{0\}$ for its generators, we can compute the winding number of $Z$ by adding the dihedral angles $w_i w_{i+1}$, taken with sign. This sum of angles is necessarily going to be a multiple $2\pi\alpha$ of $2\pi$, and the the winding number equals the integer $\alpha$.

Since each individual angle is, in absolute value, smaller than $\pi$, it is impossible to get a sum of at least $4\pi$ with four angles or less. With five angles it is possible, but not if two of them add up to less than $\pi$, as expressed by the condition in part (2). Conversely, if no three consecutive elements in $Z$ satisfy this condition, then the sum of any two consecutive angles in the cycle is at least $\pi$, the sum of four of them is at least $2\pi$, and the sum of the five of them is more than $2\pi$.
\end{proof}

Summing up, we have an easy way of checking whether a vector configuration embeds $\OvAss{k}{n}$ as a fan:

\begin{corollary}
\label{coro:fan}
Let $\VV=\{v_{ij}\}_{\{i,j\}\in \bnn} \subset \R^{k(2n-2k-1)}$ be a vector configuration. $\VV$ embeds $\OvAss{k}{n}$ as a complete fan in $\R^{k(n-2k-1)}$ if and only if it satisfies the following properties:
\begin{enumerate}
\item (Basis collection) For every facet ($k$-triangulation) $T$, the vectors $\{v_{ij}: \{i,j\}\in T\}$ are a linear basis. 
\item (ICoP) For every flip between two $k$-triangulations $T_1$ and $T_2$, the unique linear dependence among the vectors $\{v_{ij}: \{i,j\}\in T_1\cup T_2\}$ has the same sign in the two elements involved in the flip (the unique elements in $T_1\setminus T_2$ and $T_2\setminus T_1$).
\item (Elementary cycles of length $5$) Every elementary cycle of length five has three consecutive elements satisfying the sign condition in part (2) of Proposition~\ref{prop:shortcycles}.
\qedhere
\end{enumerate}
\qed
\end{corollary}

\begin{example}
	The pictures below illustrate part (3) of Corollary~\ref{coro:fan}.
	The left picture shows a flip (the union of two triangulations) belonging to the elementary cycle of the codimension-two face $\rho$ from the third picture of Example~\ref{ex:flips}.
	 In the centre picture, blue and red represent the signs of the coefficients in the circuit, that in this case is a $K_6$, for a generic vector configuration $\VV$ (see Section \ref{sec:cofactor} to understand why the signs have to be like this). The sign of 26 is opposite to 25 and 36, and this implies that, as two-dimensional vectors in $\VV/\rho\subset \R^2$, the vector 26 is a positive combination of 25 and 36, as in the right part of the figure. Thus, the angles in $(25,26)$ and $(26,36)$ add to less than $\pi$ and the cycle cannot wind twice around the origin.
	 
	 \centerline{\includegraphics[scale=0.7]{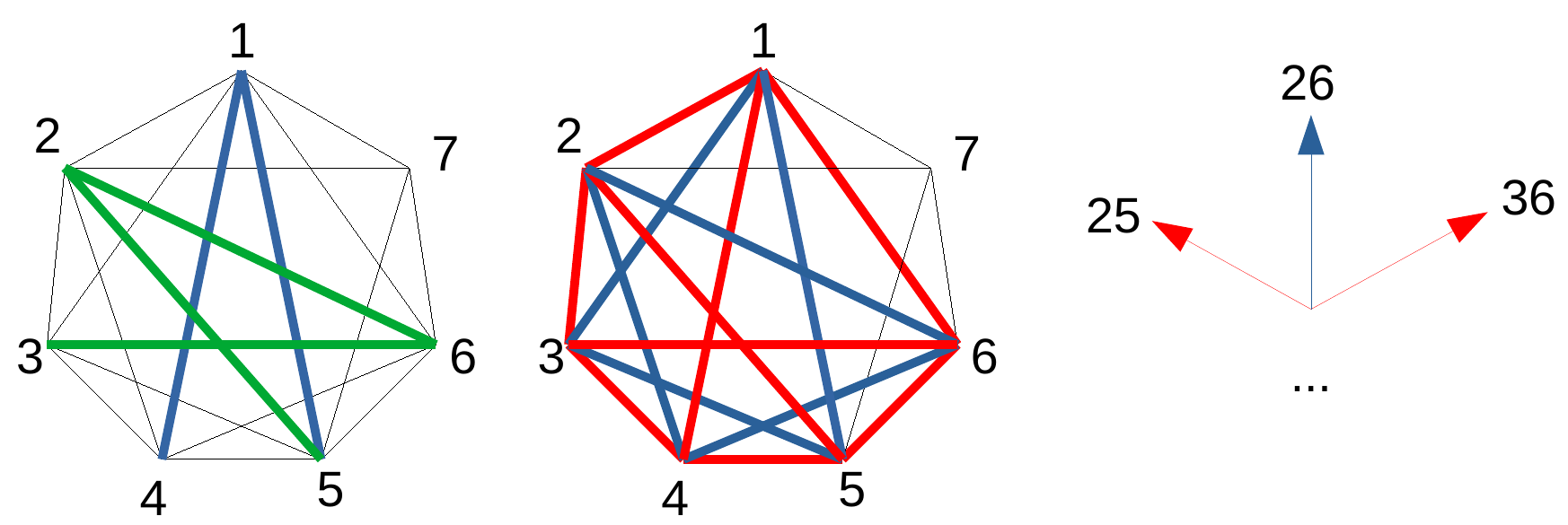}}
	\label{ex:signflip}
\end{example}

Observe that, computationally, what we need to do to apply the corollary is to check that the determinant corresponding to any $k$-triangulation is nonzero and to compute  (the signs of) the linear dependence corresponding to each flip (plus some book-keeping to identify which flips form an elementary cycle). We emphasize that only the signs are needed because computing signs may sometimes be easier than computing actual values.

\subsubsection{Conditions for polytopality}
\label{subsec:polytopal}

Once we have a collection of vectors $\VV=\{v_i: i \in v\} \subset \RR^D$ that embeds a simplicial complex $\Delta$ as a complete simplicial fan in $\R^D$, saying that the fan is the normal fan of a polytope is equivalent to saying the $\Delta$ is a \emph{regular triangulation} of $\VV$. (This is Theorem 9.5.6 in \cite{triangbook}).

Regular here means that there is a choice of lifting heights $f_i\in (0,\infty)^V$ for the vertices $i\in V$ of $\Delta$
such that $\Delta$ is the boundary complex of the cone in $\R^{D+1}$ generated by lifting the vectors $v_i\in \R^D$ to vectors $(v_i,f_i) \in \R^{D+1}$.
That is to say, we need that for every facet $\sigma \in \Delta$ the linear hyperplane containing the lift of $\sigma$ lies strictly below all the other lifted vectors.

We call such lifting vectors  $(f_i)_{i\in V}$ \emph{valid}. 
The following lemma is a version of \cite[Theorem 3.7]{RoSaSt}, which in turn is closely related to \cite[Proposition 5.2.6(i)]{triangbook}.

\begin{lemma}
\label{lemma:inequalities}
         Let $\Delta$ be a simplicial complex with vertex set $V$ and dimension $D-1$, and assume it is a 
         triangulation of a vector configuration $\VV\subset \R^D$ positively spanning $\R^D$. Then,
	a lifting vector $(f_{i})_{i\in V}$ is valid if and only if for every facet $\sigma\in \Delta$ and element $a\in V\setminus \sigma$ the inequality
	\begin{equation}\label{ineq}
		\sum_{i\in C}\omega_{i}(C)f_{i}>0
	\end{equation}
	holds, where $C=\sigma \cup\{a\}$ and $\omega(C)$ is the vector of coefficients in the unique (modulo multiplication by a positive scalar) linear dependence in $\VV$ with support in $C$, with signs chosen so that $\omega_{a}(C)>0$ for the extra element.
\end{lemma}

\begin{proof}
	This is similar to Proposition 5.2.6(i) in \cite{triangbook}. For a facet $\sigma$ and an extra element that forms a circuit $C$, we need to prove that the extra element is in the correct side via the lifting vector $f$. Let $i$ be the extra element and $f_{i}'$ the last coordinate of the intersection point of $v_i\times \R$ with the hyperplane spanned by $(v_j,f_j)_{j\in C}$. We want that $f_{i}>f_{i}'$. But obviously
	\[
	\sum_{j\in \sigma}\omega_{j}(C)f_{j}+\omega_{i}(C)f_{i}'=0
	\]
	so the condition is equivalent to $f_{i}>f_{i}'$.
\end{proof}

\begin{remark}
\label{rem:regular}
Two remarks are in order:
\begin{enumerate}
\item If we already know $\Delta$ to be a triangulation of $\VV$, it is enough to check the inequalities for the case when $i$ is a neighbor of $\sigma$, because a locally convex cone is globally convex. That is, checking validity amounts to checking one linear inequality for each ridge in $\Delta$: if 
$\tau$ is a ridge and $\tau\cup\{i\}$ and $\tau\cup\{j\}$ are the two facets containing it, we need to check inequality \eqref{ineq} for the circuit 
$C=\tau\cup\{i,j\}$ contained in $\tau\cup\{i,j\}$. (See, e.g., \cite[Theorem 2.3.20 and Lemma 8.2.3]{triangbook}).
\item If $(f_i)_{i\in V}$ is a valid lifting vector and $(w_i)_{i\in V}$ is the vector of values that a certain linear functional takes in $\VV$ then $(f_i+w_i)_{i\in V}$ is also valid. (See, e.g., \cite[Proposition 5.4.1]{triangbook}). In particular, when looking for valid vectors we can assume, without loss of generality, that $f_i=0$ for all $i$ in a certain independent set $S$ (we here say that $S$ is independent if the vectors $\{v_i\}_{i\in S}$ are linearly independent).
\end{enumerate}
\end{remark}

\subsection{Rigidity}
\label{sec:rigidity}

\label{subsec:rigidity}
Let $\p=(p_1,\dots,p_n)$ be a configuration of $n$ points in $\R^d$, labelled by $[n]$.%
\footnote{
By a \emph{configuration} we mean an ordered set of points or vectors, usually labelled by the first $n$ positive integers.
For this reason we write $\p$ as a vector rather than a set.
} 
Their \emph{bar-and-joint rigidity matrix} is the following $\binom{n}{2}\times nd$
matrix:
\begin{equation} \label{R}
R(\p):=
\begin{pmatrix}
p_1-p_2 & p_2-p_1 & 0 & \dots & 0 & 0 \\
p_1-p_3 & 0 & p_3-p_1 & \dots & 0 & 0 \\
\vdots & \vdots & \vdots & & \vdots & \vdots \\
p_1-p_n & 0 & 0 & \dots & 0 & p_n-p_1 \\
0 & p_2-p_3 & p_3-p_2 & \dots & 0 & 0 \\
\vdots & \vdots & \vdots & & \vdots & \vdots \\
0 & 0 & 0 & \dots & p_{n-1}-p_n & p_n-p_{n-1}
\end{pmatrix}.
\end{equation}
Since there is a row of the matrix for each pair $\{i,j\}\in \bnn$, rows can be considered labeled by edges in the complete graph $K_n$. The matrix is a sort of ``directed incidence matrix'' of $K_n$, except instead of having one column for each vertex $i\in[n]$ we have a block of $d$ columns, and instead of putting a single $+1$ and $-1$ in the row of edge  $\{i,j\}$ we put the $d$-dimensional (row) vectors $p_i-p_j$ and $p_j-p_i$. 

An important property of $R(\p)$ (Lemma 11.1.3 in \cite{Whiteley}) is that if the points $\p$ affinely span $\R^d$ then the rank of $R(\p)$ equals
\begin{equation}
\label{eq:rank}
\begin{cases}
	\binom{n}{2} & \text{if}\ n\leq d+1, \\
	dn-\binom{d+1}{2} & \text{if}\ n\geq d.
	\end{cases} 
\end{equation}
(Observe that the two formulas give the same result for $n\in\{d,d+1\}$.)
If the points span an $r$-dimensional affine subspace, the same formulas hold with $r$ substituted for $d$.

A pair $(G, \p)$ where $G$ is a graph on $n$ vertices and $\p$ is a set of $n$ points in $\R^n$ (positions for the vertices) is usually called a framework. 
Seeing this framework as a bar-and-joint structure, the coefficients of any linear dependence among the rows of $R(\p)$ can be interpreted as forces acting along the bars (edges) with the property that the resultant force on every vertex cancels out.
These systems of forces are called \textit{self-stresses} of \textit{equilibrium stresses}. 
We will denote $Z(R(\p))$ the space of self-stresses of $\p$. 

In the same manner, a linear dependence among the columns of $R(\p)$ can be understood as an infinitesimal motion of the vertices (that is, an assignment of velocity vectors to the joints) that preserves the length of all bars. This is called an \textit{infinitesimal flex} of the framework. 
We do not introduce a particular notation for flexes since our main interest is in the vector configuration, and matroid, of \emph{rows} of $R(\p)$. To this end, for any  $E\subset \bnn$ we denote by $R(\p)|_E$ the restriction of $R(\p)$ to the rows or elements indexed by $E$.

\begin{definition}[Rigidity]
Let $E\subset \bnn$ be a subset of edges of $K_n$ (equivalently, of rows of $R(\p)$).  We say that $E$, or the corresponding subgraph of $K_n$,  is \emph{self-stress-free} or \emph{independent} in the position $\p$ if the rows of $R(\p)|_E$ are linearly independent, and \emph{rigid} or \emph{spanning} if they are linearly spanning (that is, if they have the same rank as the whole matrix $R(\p)$).
\end{definition}

Put differently, self-stress-free and rigid graphs are, respectively, the independent and spanning sets in the linear matroid of rows of $R(\p)$. We call this matroid the \textit{bar-and-joint rigidity matroid} of $\p$ and denote it $\RR(\p)$. It is a matroid with ground set $\bnn$ and, for points affinely spanning $\R^d$, of rank given by Equation~\eqref{eq:rank}.
See, for example, \cite{GSS,Whiteley} for more information on rigidity matrices and their matroids.
Let us remark that, although rigidity theory usually deals only with $\RR(\p)$ as an (unoriented) matroid, its definition as the linear matroid of a configuration of real vectors produces in fact an \emph{oriented matroid}. Orientations will be important for us in Section \ref{sec:triangsrigid}, in the light of the results of Section~\ref{sec:polytopal}.

The following  two matrices and matroids reminiscent of $R(\p)$ are of interest:
\begin{itemize}

\item The \emph{hyperconnectivity} matroid of the configuration $\p=(p_1,\dots,p_n)$ in $\R^d$, denoted $\HH(\p)$, is the matroid of rows of
\begin{equation}\label{H}
H(\p):=
\begin{pmatrix}
p_2 & -p_1 & 0 & \dots & 0 & 0 \\
p_3 & 0 & -p_1 & \dots & 0 & 0 \\
\vdots & \vdots & \vdots & & \vdots & \vdots \\
p_n & 0 & 0 & \dots & 0 & -p_1 \\
0 & p_3 & -p_2 & \dots & 0 & 0 \\
\vdots & \vdots & \vdots & & \vdots & \vdots \\
0 & 0 & 0 & \dots & p_n & -p_{n-1}
\end{pmatrix}
\end{equation}

\item For points $\q=(q_1,\dots,q_n)$ in $\R^2$ and a parameter $d\in\N$, the $d$-dimensional \textit{cofactor rigidity}%
\footnote{This form of rigidity is usually called \emph{$C_{d-1}^{d-2}$-cofactor rigidity}, since it is related to the existence of piecewise linear splines of degree $d-1$ and of type $C^{d-2}$.}
 matroid of the points $q_1,\ldots,q_n$, which we denote $\CC_d(\q)$, is the matroid of rows of 
\begin{equation}\label{eq:C}
C_d(\q):= 
\begin{pmatrix}
\cc_{12} & -\cc_{12} & 0 & \dots & 0 & 0 \\
\cc_{13} & 0 & -\cc_{13} & \dots & 0 & 0 \\
\vdots & \vdots & \vdots & & \vdots & \vdots \\
\cc_{1n} & 0 & 0 & \dots & 0 & -\cc_{1n} \\
0 & \cc_{23} & -\cc_{23} & \dots & 0 & 0 \\
\vdots & \vdots & \vdots & & \vdots & \vdots \\
0 & 0 & 0 & \dots & \cc_{n-1,n} & -\cc_{n-1,n}
\end{pmatrix},
\end{equation}
where the vector $\cc_{ij}\in\R^d$ associated to $q_i=(x_i,y_i)$ and $q_j=(x_j,y_j)$ is
\[
\cc_{ij} := 
\left((x_i-x_j)^{d-1}, (y_i-y_j)(x_i-x_j)^{d-2},\dots,(y_i-y_j)^{d-1}\right).
\]

For $d=1$ this is independent of the choice of $\q$ and equals the directed incidence matrix of $K_n$. For $d=2$ we have $C_2(\q)=R(\q)$.
\end{itemize}

The matroids $\RR(\p)$ and $\CC_{d}(\q)$ are invariant under affine transformation of the points, and $\HH(\p)$ under linear transformation. (More generally, although we do not need this, $\RR(\p)$ and $\CC_d(\q)$ are invariant under projective transformation in $\R\proj^d$ and $\R\proj^2$ as compactifications of $\R^d$ and $\R^2$, and $\HH(\p)$ under projective transformation in $\R\proj^{d-1}$ as a quotient of $\R^d\setminus \{0\}$).
We say that the points chosen are in \emph{general position} for $\RR$ (respectively for $\CC$ or for $\HH$) if no $d+1$ of them lie in an affine hyperplane (respectively no three of them in an affine line or no $d$ of them in a linear hyperplane). In the three cases, general position implies that the corresponding matroid has the rank stated in Equation~\eqref{eq:rank} and that every copy of the graph $K_{d+2}$ is a circuit.  Nguyen~\cite{Nguyen} showed that the matroids on $\bnn$ with these properties are exactly the  \textit{abstract rigidity matroids} introduced by Graver in \cite{Graver}.

Clearly, in the three cases and for each choice of the ``dimension'' $d$ there is a unique most free matroid that can be obtained, in the sense that the independent sets in any other matroid will also be independent in this one, realized by sufficiently generic choices of the points. We call these the \emph{generic bar-and-joint, hyperconnectivity, and cofactor matroids of dimension $d$ on $n$ points}, and denote them $\RR_d(n)$, $\HH_d(n)$, $\CC_d(n)$. 
(Observe, however, that this generic matroid may stratify into several different generic oriented matroids; this is important for us since we will be concerned with the signs of circuits, by the results in Section~\ref{sec:polytopal}).

In \cite{CreSan:moment} we prove that these three rigidity theories coincide when the points $\p$ or $\q$ are chosen along the moment curve (for bar-and-joint and hyperconnectivity) and the parabola (for cofactor). More precisely:

\begin{theorem}[\cite{CreSan:moment}] 
\label{thm:3matroids}
Let $t_1< \dots < t_n\in \R$ be real parameters. Let 
\[
p_i=(1,t_i,\dots,t_i^{d-1})\in \R^d, \quad
p'_i=(t_i,t_i^2,\dots,t_i^d)\in \R^d, \quad
q_i=(t_i,t_i^2)\in \R^2.
\]
Then, the three matrices $H(p_1,\dots,p_n)$, $R(p'_1,\dots,p'_n)$ and $C_d(q_1,\dots,q_n)$ can be obtained from one another 
multiplying on the right by a regular matrix and then multiplying rows by some positive scalars.

In particular, the rows of the three matrices define the same oriented matroid. 
\end{theorem}

\begin{proof}
This follows from the proofs of Lemma 2.3 and Theorem 2.5 in \cite{CreSan:moment}. Although the statements there speak only of the \emph{matroids} of rows,  the proofs show that dividing each row $(i,j)$ of $R(p'_1,\dots,p'_n)$ by $t_j-t_i$ and that of $C_d(q_1,\dots,q_n)$ by $(t_j-t_i)^{d-1}$ one obtains matrices that are equivalent to  $H(p_1,\dots,p_n)$ under multiplication on the right by a regular matrix.
\end{proof}

\begin{definition}[Polynomial rigidity]
\label{def:poly}
We call the matrix $H(p_1,\dots,p_n)$ in the statement of Theorem~\ref{thm:3matroids} the \emph{polynomial $d$-rigidity matrix with parameters $t_1,\dots,t_n$}. We denote it $P_d(t_1,\dots,t_n)$, and denote $\PP_d(t_1,\dots,t_n)$ the corresponding matroid.
\end{definition}

Among the polynomial rigidity matroids $\PP_d(t_1,\dots,t_n)$ there is again one that is the most free, obtained with a sufficiently generic choice of the $t_i$. We denote it $\PP_d(n)$ and call it the \emph{generic polynomial $d$-rigidity matroid on $n$ points}. Theorem~\ref{thm:3matroids} implies that we can regard $\PP_d(n)$ as capturing generic bar-and-joint rigidity along the moment curve, generic hyperconnectivity along the moment curve, or generic cofactor rigidity on a conic.

We do not know whether $\PP_d(n)$ equals $\HH_d(n)$, but we do know that $\HH_d(n)$, $\RR_d(n)$ and $\CC_d(n)$ are different for $d\ge 4$ and $n$ large enough. For example:
\begin{itemize}
\item $K_{d+1,d+1}$ is a circuit in $\PP_d(n)$ and $\HH_d(n)$, but independent in $\RR_d(n)$ and $\CC_d(n)$ for every $d\ge 2$~(More strongly, $K_{d+1,\binom{d+1}2}$ is a basis in both, \cite[Theorem 9.3.6 and Example 11.3.12]{Whiteley}).
\item $K_{6,7}$ is a basis in $\CC_4(n)$ and dependent in $\RR_4(n)$ for $n\ge 13$~\cite[Sections 11.4 and 11.5]{Whiteley}. (The example generalizes to show that $\CC_d(n)\ne \RR_d(n)$ for $n-9\ge d\ge 4$).
\end{itemize}
See \cite{CreSan:moment} for a recent account of these and other relations among these matroids, including some questions and conjectures. See~\cite{NSW21} for a comprehensive study of bar-and-joint and cofactor rigidities.

\section{Obstructions to realizability with cofactor rigidity}
\label{sec:obstructions}

Our main goal in this paper is to study whether one of the three forms of rigidity from Section~\ref{sec:rigidity} provides, by choosing the configurations $\p$ in $\R^{2k}$ or $\q$ in $\R^2$ adequately, realizations the $k$-associahedron $\OvAss{k}{n}$. 
For positive results (realizations) the strongest possible setting goes via the polynomial rigidity of Definition~\ref{def:poly}, since that is a special case of the other three. For negative results (obstructions to realization) we are going to use cofactor rigidity. This is stronger than using polynomial rigidity, and is also the most natural setting for studying $k$-associahedra since, after all, the combinatorics of a $k$-associahedron comes from thinking about crossings in the complete graph embedded with vertices in convex position in the plane.

\subsection{Some results on cofactor rigidity}\label{sec:cofactor}

In this section we present some  results about cofactor rigidity that we need later. 

We first show that cofactor rigidity is invariant under projective transformation. This, as some other results from this section, was already noticed by Whiteley~\cite{Whiteley}, but we develop things from scratch since we will not only be interested in the cofactor rigidity \emph{matroid} but also in the \emph{oriented matroid}.
Notice also that the same projective invariance of the matroid is true and well-known for bar-and-joint rigidity (see again~\cite{Whiteley}). 

Throughout this section we work primarily with a vector configuration $\Q=(Q_1,\dots,Q_n)$ in dimension three, that is, with $Q_i=(X_i,Y_i,Z_i) \in \R^3\setminus \{0\}$. We normally assume that $\Q$ is in general position (every three of its vectors form a linear basis) and sometimes that it is also in \emph{convex position}: (a) each of the vectors  $Q_i$ generates a ray of 
\[
\cone(\Q) = \{\sum_i\lambda_i Q_i : \lambda_i \ge 0\}, 
\]
and all these rays are different, and (b) the cyclic order of $Q_1,\dots,Q_n$ equals their order as rays of $\cone(\Q)$.

In this setting, let us redefine the vectors $\cc_{ij}$ that appear in the matrix $C_d(\q)$ in terms of the vectors $Q_i$ as follows. We let
\[
\cc_{ij}=
\left(x_{ij}^{d-1}, y_{ij}x_{ij}^{d-2},\dots,y_{ij}^{d-1}\right),
\]
{ where }
\[
x_{ij}=X_iZ_j-Z_iX_j, \quad y_{ij}=Y_iZ_j-Z_iY_j.
\]
We define the matrix $C_d(\Q)$ exactly as in Equation~\eqref{eq:C}, but with these new vectors $\cc_{ij}$.
Observe that the original definition of $C_d(\q)$ is a special case of this one, obtained when we take all the $Z_i$'s equal to $1$ and we let $q_i=(X_i,Y_i)$. 

With the new definition, we have the following invariance:

\begin{proposition}
\label{prop:invariance}
Let $\Q=(Q_1,\dots,Q_n)$ be a vector configuration in $\R^3\setminus \{0\}$.
Then, 
\begin{enumerate}
\item The column-space of $C_d(\Q)$, hence the oriented matroid $\CC_d(\Q)$ of its rows,  is invariant under a nonsingular linear transformation of $\Q$.
\item The matroid $\CC_d(\Q)$ is also invariant under rescaling (that is, multiplication by non-zero scalars) of the vectors $Q_i$. If the scalars are all positive or $d$ is odd then the same holds for the oriented matroid.
\end{enumerate}
\end{proposition}

\begin{proof}
For each vector $Q\in \R^3\setminus \{0\}$ let $C_{d-1}^{d-2}(Q)$ be the set of all three-variate polynomials in $\R[X,Y,Z]$ that are homogeneous of degree $d-1$ and such that all their partial derivatives up to order $d-2$ vanish at $Q$. This is a vector space of dimension $d$.
In fact, if we fix a $Q_i=(X_i,Y_i,Z_i)$ and consider $Q_j=(X,Y,Z)$ as a vector of variables, then the $d$ entries in the vector $\cc_{ij}$ are a basis for the space $C_{d-1}^{d-2}(Q_i)$. In particular, the $i$-th block in the matrix $C_d(\Q)$ has as rows the vectors obtained by evaluating that basis of $C_{d-1}^{d-2}(Q_i)$ either at $0$ (if the row does not use the point $i$) or at one of the $Q_j$'s (if the row corresponds to the edge $\{i,j\}$). 

Now, let $\Q=(Q_1,\dots,Q_n)$.
A nonsingular linear transformation $L:\R^3\to \R^3$ induces, for each vector $Q_i$, a linear map $\tilde L_i$ from the space $C_{d-1}^{d-2}(L(Q_i))$ to the space $C_{d-1}^{d-2}(Q_i)$, defined by $\tilde L_i(f) = f \circ L$. 

Let $M_i\in \R^{d\times d}$ be the matrix of $\tilde L_i$ in the bases of $C_{d-1}^{d-2}(L(Q_i))$ and  $C_{d-1}^{d-2}(Q_i)$ described above. Let $M\in \R^{dn\times dn}$ be the block-diagonal matrix with blocks of size $d\times d$ and having in the $i$-th diagonal block the matrix $M_i$. Then we have that
\[
C_d(L(\Q)) = C_d(\Q) M^{-1}.
\]
As $L$ is nonsingular, this proves part (1).

For part (2):  the effect of multiplying a $Q_i$ by a scalar $\lambda_i$ is to multiply all the rows of edges using $i$  by the scalar $\lambda_i^{d-1}$. Hence, although the column space $C_d(\Q)$ changes by rescaling, the matroid $\CC_d(\Q)$ does not, and the oriented matroid does not either as long as the rescaling factors are all positive or $d$ is odd. 
\end{proof}

We now translate the above result to the original setting of a point configuration $\q=(q_1,\dots,q_n)$ in $\R^2$:

\begin{corollary}
The matroid $\CC_d(\q)$ is invariant under projective transformation of $\q$. If $d$ is odd or the projective transformation sends $\conv(\q)$ to lie in the affine chart of $\RP^2$ (the subset of the projective points $[X:Y:Z]$ with $Z\ne 0$), the same is true for the oriented matroid.
\end{corollary}

\begin{proof}
Starting with a point configuration $\q=(q_1,\dots,q_n)$ in the (affine) plane we can consider the vector configuration $\Q=(Q_1,\dots,Q_n)$ with  $Q_i=(q_i,1)\in \R^3\setminus \{0\}$. A projective transformation in $\q$ amounts to a linear transformation in $\Q$. Moreover, if the projective transformation sends $\conv(\q)$ to lie in the affine chart of $\RP^2$ then all the $Z'_i$ in the transformed vector configuration are positive, so they can be brought back to the form $(x,y,1)$ by a positive rescaling.
\end{proof}

Our next result is essentially~\cite[Theorem 11.3.3]{Whiteley} and shows how  examples and properties of cofactor rigidity in dimension $d$ can be lifted to dimension $d+1$ via coning. 
Recall that the contraction $\p/p_i$ of a vector configuration at an element $p_i$ was defined in Section \ref{sec:polytopal}.

\begin{proposition}[Coning Theorem, \protect{\cite[Theorem 11.3.3]{Whiteley}}]
	\label{prop:coning-cofactor}
	Let $\Q=(Q_1,\dots$, $Q_{n+1})$ be a vector configuration in general position in $\R^3$.
	Then, $\CC_d(Q_1,\dots,Q_n)$ is the contraction to $\binom{[n]}{2}$ of the matroid 
	$\CC_{d+1}(\Q)$. If the vectors are in convex position, the same is true for the oriented matroids.
\end{proposition}

Let us mention that  the same result holds for the other two forms of rigidity, $\RR$ and $\HH$~\cite[Theorem 9.3.11]{Whiteley}, and \cite[Theorem 5.1]{Kalai}.
We call this statement ``coning theorem'' because it implies that a graph $G$ with vertex set $[n]$ is $d$-independent or $d$-rigid when realized on $(Q_1,\dots,Q_n)$ if and only if its cone $G*\{n+1\}$ is $(d+1)$-independent or $(d+1)$-rigid on $(Q_1,\dots,Q_n, Q_{n+1})$. Here, the \emph{cone over a graph} $G=([n],E)$ is defined as the graph with vertex set $[n+1]$ and with edges 
\[
E*\{n+1\}:=E\cup \left\{\{i,n+1\}: i\in [n]\right\}.
\]

\begin{proof}
By a linear transformation we can assume without loss of generality that $Q_{n+1} = (0,1,0)$ and that no other $Q_i$ lies in the ``hyperplane at infinity'' $\{Z=0\}$; hence, we can rescale them to have $Z_i=1$ for $i=1,\dots,n$.
	This linear transformation and rescaling do not affect the matroids. Moreover, if the original vectors are in convex position, all of them are in a half-space whose delimiting plane contains $Q_{n+1}$. This implies that  we can further assume that the linear transformation sends this plane to $Z_i=0$ and after this step $Z_i>0$ for every $i$, so that the rescaling is positive and does not affect the oriented matroids either.
	
Under these assumptions we have that
	\[
	\cc_{i,n+1} =(0,0,\dots, (-1)^{d-1}).
	\]
In particular, the contraction of the elements $\{i,n+1\}$ in the matroid $\CC_{d+1}(\Q)$ can be performed in the matrix $C_{d+1}(\Q)$ as follows: first, forget the last block of columns (the one corresponding to $Q_{n+1}$). This does not affect the oriented matroid since the sum of the $n$ blocks of $\CC_{d+1}(\Q)$ equals zero (that is, the columns in each one block are linear combinations of the other blocks).
After the block of $Q_{n+1}$ is deleted, the rows $\{i,n+1\}$ that we want to contract have a single non-zero entry, so the contraction is equivalent to deleting those rows and their corresponding columns, namely the last column in the block of each $Q_i$, $i=1,\dots,n$. 

The resulting matrix coincides with 
$C_{d}(Q_1,\dots, Q_{n})$ except that the row corresponding to each edge $\{i,j\}$ has been multiplied by the factor $x_{ij}:= X_iZ_j-Z_iX_j = X_i-X_j$. General position (under the assumption $Q_{n+1}=(0,1,0)$ and $Z_i=1$ for every other $i$) implies $X_i\ne X_j$ for $i\ne j$, so this factor $x_{ij}$ does not affect the matroid.

The factor could a priori affect the oriented matroid, but our assumption that the vectors are in convex position with $Q_{n+1}=(0,1,0)$ and with $Z_1=\dots=Z_n=1$ implies that  $X_1< \dots<X_n$. Hence, the spurious factors  $x_{ij}$ are all of the same sign (all negative) and do not change the oriented matroid. 
\end{proof}

We now look at what happens if the point we add/delete is not the last one $Q_{n+1}$ but an intermediate one $Q_i$. This is a mere cyclic reordering of the points with respect to the previous result, but reordering has a non-trivial effect in the cofactor matrix, because of a lack of symmetry in its definition. Indeed, the row of an edge $\{i,j\}$ with $i < j$ has the  shape
\[
(\dots, \cc_{ij}, \dots, -\cc_{ij}, \dots).
\]
If the reordering keeps $i$ before $j$  the row does not change; its entries simply get moved around as indicated by the reordering. In contrast, if after reordering we end up having $j$ before $i$ then the new row equals
\[
(\dots, \cc_{ji}, \dots, -\cc_{ji}, \dots).
\]
That is, we get $\cc_{ji}$ where the ``moving around'' should give $-\cc_{ij}$ and $-\cc_{ji}$ where we should get $ \cc_{ij}$. The effect of this depends on the parity of $d$. If $d$ is even, then $\cc_{ji} = -\cc_{ij}$ and the relabelling does not affect the oriented matroid. If $d$ is odd, however, then $\cc_{ji} = \cc_{ij}$, so the relabelling globally changes the sign of that row of the matrix. This implies:

\begin{proposition}
	\label{prop:coning-cofactor2}
	Let $\Q=(Q_1,\dots,Q_{n+1})$ be vectors in $\R^3$ in general position. 
	Then the oriented matroid $\CC_d(Q_1,\ldots,Q_{i-1}$, $Q_{i+1},\ldots,Q_{n+1})$ is obtained from $\CC_{d+1}(\Q)$ by contracting at the elements $\{i,j\}$ with $j\in [n+1]\setminus \{i\}$, and reorienting the elements $\{j,k\}$ with $1\le j<i<k\le n+1$.
\end{proposition}

\begin{proof}
	Let us first relabel points cyclically  so that the point $i$ becomes $n+1$, then apply Proposition~\ref{prop:coning-cofactor}, and finally relabel the points back to their original labels. As noted above, relabelling does nothing if the dimension is even or to the edges that keep their direction, but it reorients the edges that change their direction (that is, the edges $\{j,k\}$ with $j<i<k$) if the dimension is odd. Since we are relabelling first in dimension $d+1$ and then in dimension $d$, exactly one of them is odd.
\end{proof}

Now we prove a result about the number of sign changes in any dependence in $\CC_d(\Q)$, with elements in convex position:

\begin{lemma}\label{lemma:chsign}
	Let $\Q=(Q_1,\dots,Q_n)$ be vectors in convex position. Let $\lambda\in \R^{\binom{n}{2}}$ be a linear dependence among the rows of $C_d(\Q)$.
For each $i$, considered the sequence formed by $\{\lambda_{ij}\}_{j\ne i}$, with values of $j$ ordered cyclically from $i$. That is, with the order $(i+1,i+2,\dots,n,1,\dots,i-1)$. Then:
\begin{enumerate}
\item If $d$ is even, the sequence changes sign at least $d$ times.
\item If $d$ is odd, the same happens for the sequence $\{-\lambda_{ij}\}_{j>i}\cup \{\lambda_{ij}\}_{j<i}$.
\end{enumerate}
\end{lemma}

\begin{proof}
Let us first assume that $d$ is even. In this case, as mentioned above, a cyclic relabelling does not change the oriented matroid, so we can assume without loss of generality that $i=n$. Also, by linear transformation and positive rescaling we can assume that  $Q_n=(0,0,1)$ and that $Z_j=1$ and $X_j >0$ for $j\ne i$. 
Observe that under these assumptions we have
\[
\cc_{jn} = (X_j^{d-1},  X_j^{d-2} Y_j,  \dots,  Y_j^{d-1}) = X_j^{d-1} (1, m_j,\dots, m_j^{d-1}),
\]
where $m_j:= Y_j/X_j$ is the slope of the segment from $q_n=(0,0)$ to $q_j=(X_j,Y_j)$. Since the $X_j$ are positive we can neglect them without changing the oriented matroid, and convex position implies that the $m_j$ are increasing: $m_1 < \dots < m_{n-1}$.

Hence, the sequence $(\lambda_{jn})_{j\in [n-1]}$ that we want to study is (at least regarding its signs) a linear dependence among the vectors $(1, m_j,\dots, m_j^{d-1})$ for an increasing sequence of $m_j$'s. Put differently, it is an affine dependence among the vertices of a cyclic $(d-1)$-polytope. It is well-known that the circuits in the cyclic polytope are alternating sequences with $d+1$ points~\cite[Theorem 6.1.11]{triangbook}, hence they have $d$ sign changes. Since every dependence is a composition of circuits~\cite[Lemma 4.1.12, Corollary 4.1.13]{triangbook}, hence it has at least the same number of changes.

For the case where $d$ is odd all of the above remains true except the initial cyclic relabelling reverts the sign of all the $\lambda_{ij}$ with $j>i$.
\end{proof}

\subsection{The Morgan-Scott obstruction in cofactor rigidity}
\label{sec:desargues}

In this section we show that the graph obtained from $K_6$ by removing a perfect matching (that is, the graph of an octahedron) is  a circuit or a basis in the three-dimensional cofactor rigidity $\CC_3$, depending on whether the points are in ``Desargues position'' or not. This is well-known in the theory of splines, and usually called the \emph{Morgan-Scott split} or Morgan-Scott configuration~\cite{MorganScott}. 
We here rework it, following \cite[Example 4]{Whiteley-splines}, since we need an oriented version of it. See also \cite[Example 41, p. 90]{NSW21}.

\begin{definition}[Desargues position]
Let $\q=(q_1,\dots,q_6)$ be a configuration of six points in convex position in the plane. Let us call upper side of the line $25$ the side containing the points $1$ and $6$, and lower side the other one. We say  that $\q$ is \emph{positively (resp. negatively) oriented} if the intersection of the lines $14$ and $36$ lies in the lower (resp. upper) side of $25$. We say that $\q$ is in \emph{Desargues position} if none of the two happens, that is, if the lines $14$, $25$ and $36$ are concurrent.
\end{definition}

See Figure~\ref{fig:desargues} for an illustration of this concept, with points chosen along the standard parabola.
We call the concurrent case \emph{Desargues position} since Desargues theorem says that this concurrency is equivalent to the triangles $q_1q_3q_5$ and $q_2q_4q_6$ being axially perspective.

\begin{figure}[htb]
\includegraphics[width=6.5cm]{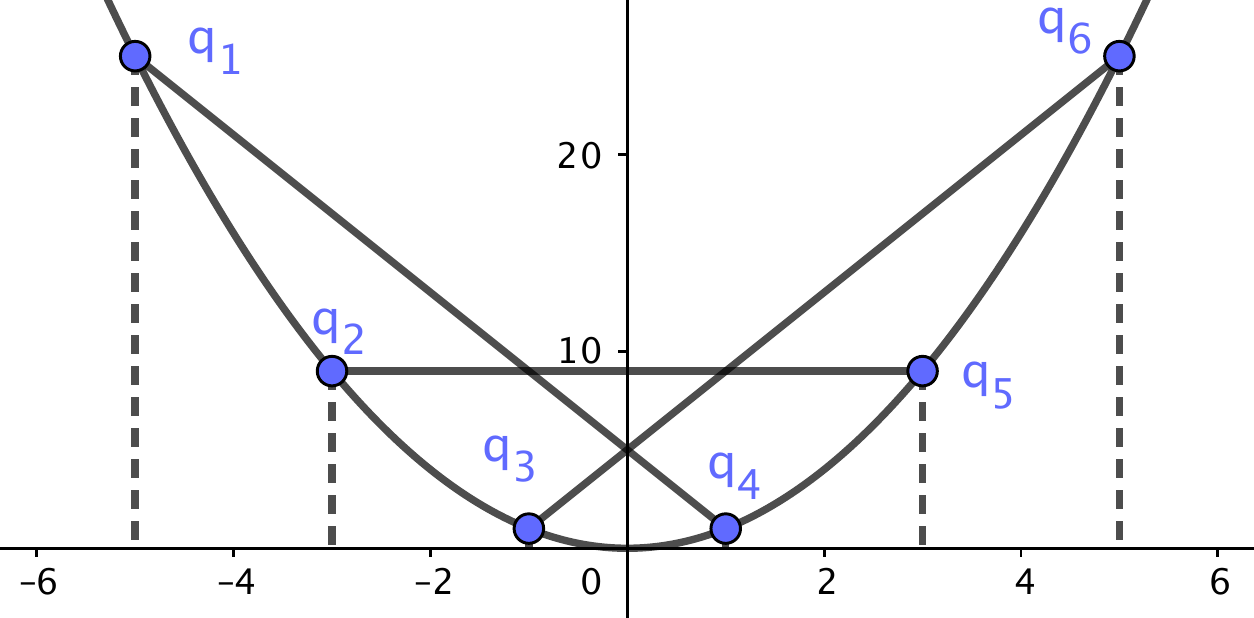}
\qquad
\includegraphics[width=5cm]{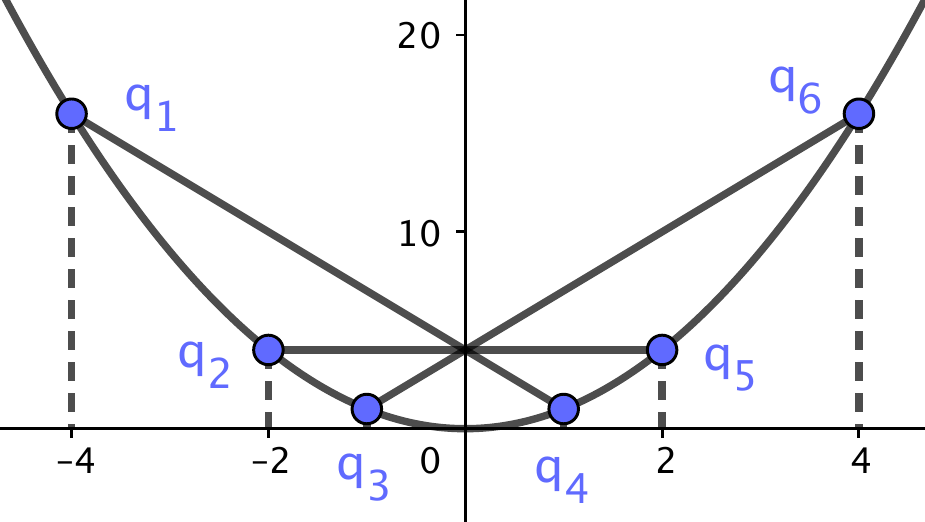}
\caption{Two configurations of six points in convex position, chosen along the parabola. The configuration in the left is positively oriented, the one on the right is in Desargues position}

\label{fig:desargues}
\end{figure}

\begin{theorem}
\label{thm:cofact}
	Consider the graph $G= K_6\setminus\{25,36\}$ embedded with six points $\q$ in general position. Then, 
	\begin{enumerate}
	\item $G$ is spanning in  $\CC_3(q)$, hence it contains a unique dependence. This dependence may not vanish at any edge other than $14$.
	\item Assume $\q$ in convex position, and let $(\lambda_{ij})_{ij}\in \R^{\binom{6}{2}}$ be this unique dependence. Then $\lambda_{15}\ne 0$ and the sign of $\lambda_{14}/\lambda_{15}$ is positive, negative, or zero, if $\q$ is positively oriented, negatively oriented, or in Desargues position, respectively.
	\end{enumerate}
\end{theorem}

This statement immediately implies:

\begin{corollary} 
	\label{coro:cofact}
	Consider the graph $G= K_6\setminus\{14,25,36\}$ embedded with six points $\q$ in convex position. 
	Then, $G$ is a circuit in $\CC_3(q)$ if the points are in Desargues position, and a basis otherwise.
\end{corollary}

The proof of the first part of Theorem~\ref{thm:cofact} is easy. Let $G'=G\setminus \{ij\}$ for an edge $ij$ different from $14$. Without loss of generality assume $i\not\in\{1,4\}$. Then $G'$ has degree three at vertex $i$ and $G'\setminus i$ equals $K_5$ minus one edge. Since $K_5$ is a circuit in $\CC_3$ (for any choice of points in general position), $G'\setminus i$ is a basis, and hence $G'$ is a basis too.  In particular, $G$ is spanning and contains a unique circuit, and this circuit does not vanish at the edge $ij$.

To prove part two, we follow the derivation in \cite[Example 4]{Whiteley-splines}. There, the following concept is introduced as a way to express  the determinant of a submatrix in the cofactor matrix $C_3$ of a triangulated sphere.

\begin{definition}[3-fan]
	Given a graph $G=([n],E_0)$ and a bipartition of the vertices $[n]=V_0\cup V_1$, a \emph{3-fan} in $(V_0,V_1,E_0)$ is an orientation of $G$ such that the vertices in $V_0$ have out-degree 3 and those in $V_1$ have out-degree 0.
	
	For a 3-fan $\pi$ and a vertex $i\in V_0$, let $\pi^i$ be the set of three edges that start at the vertex $i$. The \emph{sign of $\pi$}, denoted $\sigma(\pi)$, is the sign of $(\pi^1,\pi^2,\ldots)$ as a permutation of $E_0$, with the three elements of each $\pi^i$ in increasing order, multiplied by $(-1)^r$ where $r$ is the number of edges oriented from a vertex to another with lower index.
\end{definition}

In what follows, we will denote by $C_d(\q)|_{(E,V)}$ the restriction of the cofactor matrix of $\q$ to the rows indexed by $E$ and the column blocks indexed by $V$.

\begin{definition}[Notation \protect{$[q_i;q_jq_kq_l]$}]
	For $\q:V\to\R^2$, we define $[q_i;q_jq_kq_l]=\det C_3(\q)|_{(\{(i,j),(i,k),(i,l)\},i)}$. This determinant can be shown to be equal to $|q_iq_jq_k|\cdot|q_iq_jq_l|\cdot|q_iq_kq_l|$, where $|abc|$ denotes the determinant of the three points (written as $(x,y,1)$).
\end{definition}

The following statement summarizes the derivations in \cite[pp.15--17]{Whiteley-splines}:

\begin{lemma}\label{lemma:detfans}
	Let $(V,E)$ be the graph of a triangulated sphere and $\q:V\to\R^2$ a position for the vertices that realizes the graph as planar. Let $V_0$ the set of internal vertices, $V_1$ the three external vertices and $E_0$ the internal edges of the graph.
	Then
	\[\det C_3(\q)|_{(E_0,V_0)}=\sum_{\pi\text{ 3-fan in }\{V_0,V_1,E_0\}}\sigma(\pi)[\pi^1][\pi^2]\ldots[\pi^{n-3}]\]
	where $[\pi^i]$ stands for $[q_i;q_jq_kq_l]$ with $\pi^i=\{(i,j),(i,k),(i,l)\}$.
\end{lemma}

With this we can finish the proof of Theorem \ref{thm:cofact}:

\begin{proof}[Proof of Theorem \ref{thm:cofact}]
	The coefficients of a row dependence in an $(N+1)\times N$ matrix are the complementary minors of each row with alternating signs. In our case, our initial matrix $C_3(\q)$ has size $13\times 18$ and rank 12, but we can by an affine transformation fix the positions of vertices $1$, $2$ and $3$ (which implies no loss of generality) and then delete the nine columns of their three blocks, plus the three rows of the triangle they form. This leaves us with a $10\times 9$ matrix whose unique row-dependence we want to study.
	The coefficients  $\lambda_{14}$  and $\lambda_{15}$  have the same sign in the dependence if and only if their complementary minors have opposite sign. 
	
	To compute these two signs we use Lemma \ref{lemma:detfans} with $V_0=\{4,5,6\},V_1=\{1,2,3\}$.
	
	For the edge 15 this is easy because the remaining edges form $K_6\setminus\{12,13,23,15$, $25,36\}$, in which the only possible 3-fan is $\{41,42,43,53,54,56,61,62,64\}$. This is an even permutation in which there are 8 ``reversed'' edges, hence the sign of the 3-fan is positive.
	
	By  Lemma \ref{lemma:detfans}, the determinant is
	\[
	[4;123][5;346][6;124]=|412|\,|413|\,|423|\,|534|\,|536|\,|546|\,|612|\,|614|\,|624|
	\]
	where $i$ stands for $q_i$. A determinant of three points is positive if they are ordered counter-clockwise and negative otherwise. In this case the result is positive because there are two negative determinants, 536 and 546. (Here and in the rest of the proof we assume without loss of generality that our points are not only in convex position but also placed in counter-clockwise order along their convex hull).
	
	Now we  compute the determinant for $K_6\setminus\{12,13,23,14, 25,36\}$. There are two 3-fans here:
	\[
	\{42,43,46,51,53,54,61,62,65\}\text{ and }\{42,43,45,51,53,56,61,62,64\} 
	\]
	Both are even permutations, the first one with 8 reversed edges and the second with 7. By Lemma \ref{lemma:detfans} the determinant is 
	\begin{align*}
	&[4;236][5;134][6;125]-[4;235][5;136][6;124] =\\
	=&|423|\,|426|\,|436|\,|513|\,|514|\,|534|\,|612|\,|615|\,|625|  \\ &\qquad\qquad\qquad -   |423|\,|425|\,|435|\,|513|\,|516|\,|536|\,|612|\,|614|\,|624|\\
	=&|126|\,|135|\,|145|\,|156|\,|234|\,|246|\,|256|\,|345|\,|346|   \\ &\qquad\qquad\qquad-   |126|\,|135|\,|146|\,|156|\,|234|\,|245|\,|246|\,|345|\,|356|\\
	=&|126|\,|135|\,|156|\,|234|\,|246|\,|345|\,(|145|\,|256|\,|346|-|146|\,|245|\,|356|)
	\end{align*}
	
	The  factor $|126|\,|135|\,|156|\,|234|\,|246|\,|345|$ is positive: since the points are in convex counter-clockwise position every determinant $|abc|$ with $a<b<c$ is positive. Hence, we ignore it.
	To further simplify the rest we use the Pl\"ucker relations
	\[
	|145|\,|256| = |125|\,|456|+|245|\,|156|, \qquad
	|146|\,|356| = |346|\,|156|-|546|\,|136|.
	\]
	Hence, the last factor becomes:
	\begin{align*}
	&|145|\,|256|\,|346|-|146|\,|245|\,|356|\\
	=&|125|\,|456|\,|346|+|245|\,|156|\,|346|-|346|\,|245|\,|156|+|546|\,|245|\,|136| \\
	=&|456|\,(|125|\,|346|-|136|\,|245|)
	\end{align*}
	Dividing again by the positive factor $|456|$ and by $|245|\,|346|$ we get that the sign of the determinant equals
	\[
	\frac{|125|}{|245|}-\frac{|136|}{|346|} =
	\frac{|120|}{|240|}-\frac{|130|}{|340|},
	\]
	where we call $q_0$ (and abbreviate as $0$) the intersection point of 25 and 36. This last expression can be rewritten in term of the angles around $q_0$, as follows:
	\begin{equation}\label{sines}
	\frac{|120|}{|240|}-\frac{|130|}{|340|}=
	\frac{\sin\angle 201}{\sin\angle 402}-\frac{\sin\angle 301}{\sin\angle 403}=
	\frac{\sin\alpha}{\sin\beta}-\frac{\sin\alpha'}{\sin\beta'},
	\end{equation}
where $\alpha$, $\alpha'$, $\beta$ and $\beta'$ are displayed in Figure~\ref{fig:posoriented}.
\begin{figure}[htb]
	\includegraphics[width=6.5cm]{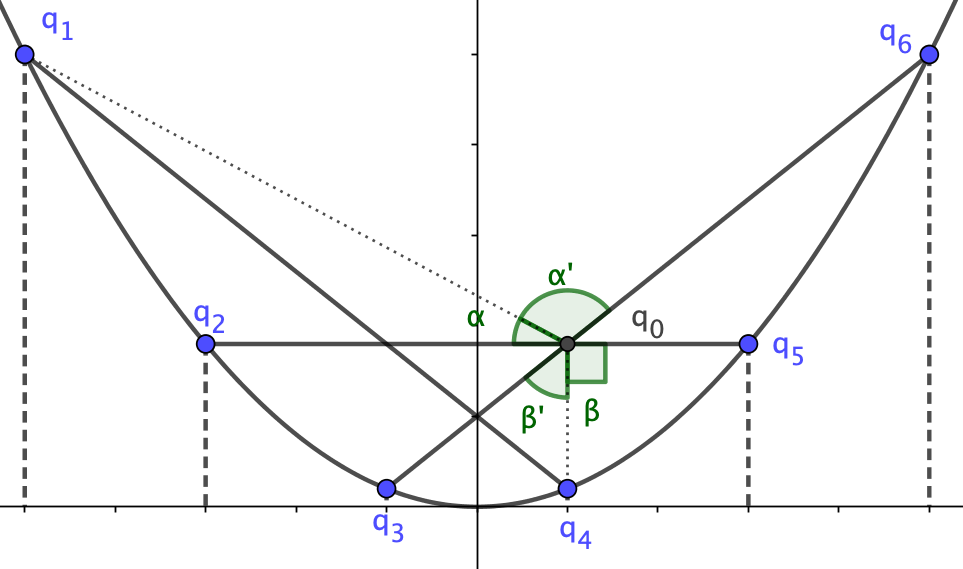}
	\caption{Explanation of the last equivalence in the proof of Theorem \ref{thm:cofact}. Each of the inequalities $\alpha<\beta$ and $\alpha'>\beta'$ is equivalent to the configuration being positively oriented.}
	\label{fig:posoriented}
\end{figure}

Looking at the figure we see that the configuration is positively oriented if, and only if, $\alpha<\beta$ and $\alpha'>\beta'$, and it is negatively oriented if the opposite inequalities hold.
	Hence:
		\begin{align*}
		\lambda_{14}\lambda_{15}>0 & \Leftrightarrow \text{the complementary determinants have opposite sign} \\
		& \Leftrightarrow \text{the complementary determinant to 14 is negative} \\
		& \Leftrightarrow \text{the value of \eqref{sines} is negative} \\
		& \Leftrightarrow \alpha<\beta \text{ and } \alpha'>\beta' \\
		& \Leftrightarrow \text{the configuration is positively oriented.}
		\qedhere
	\end{align*}
\end{proof}

Consider now the graph $K_6\setminus\{14,25,36\}$, with vertices in convex position. That is, $K_6$ minus a perfect matching, which is the graph of an octahedron. Since this graph equals $G\setminus\{14\}$, Corollary~\ref{coro:cofact} implies that $K_6\setminus\{14, 25, 36\}$ is a circuit in $\CC_3$  in Desargues position, and a basis in non-Desargues position. In particular:

\begin{corollary} 
	\label{coro:octahedron}
	The graph $K_6$ without the matching $\{14, 25, 36\}$ is a circuit in the polynomial rigidity matroid  $\PP_3(1,3,4,6,7,9)$, and a basis in the generic matroid $\PP_3$.
\end{corollary}

\begin{proof}
	Applying a translation to the parameters $t_i$ along the parabola produces an affine transformation of the point configuration, hence it does not change the oriented matroid $\PP$. So, the statement for $\PP_3(1,3,4,6,7,9)$ is equivalent to the same statement for the more symmetric $\PP_3(-4,-2,-1,1,2,4)$. The latter is in Desargues position, as seen in Figure~\ref{fig:desargues}.
\end{proof}

We can now prove that, for $k=3$ and $n=9$, there are positions where the rows of the cofactor matrix do not realize the multiassociahedron as a basis collection:

\begin{proof}[Proof of Theorem~\ref{thm:Desargues}]
We start with the graph $K_6- \{16, 37, 49\}$, with vertices labelled $\{1,3,4,6,7$, $9\}$.
Its coning at three vertices labelled $2$, $5$ and $8$ is the graph $K_9 - \{16, 37, 49\}$. The statement in the first sentence then follows from 
Proposition \ref{prop:coning-cofactor2} and Corollary \ref{coro:cofact}. The second sentence follows from 
Corollary~\ref{coro:octahedron}.
\end{proof}

Even more strongly, we can show that in this ``realization'' not even the map $\phi_{\OvAss{3}{9},\VV}:
|\OvAss{3}{9}|\to S^5$ of Theorem \ref{thm:winding} is well-defined:

\begin{example}\label{ex:desargues}
	Let $T$ be the 3-triangulation $K_9\setminus\{16,37,49\}$, and locate it with a $\p$ in Desargues position, so that it contains a circuit. For example, embedding it via $P(1,2,3,4,5,6,7,8,9)$ in the parabola. By Theorem \ref{thm:Desargues} $T$ is a circuit in this embedding.
	The vertices 1, 3, 4, 6, 7 and 9 have degree 7, so the edges incident to each of them must have alternating signs in the circuit by Lemma~\ref{lemma:chsign}. As a result, the six relevant edges in $T$, which are $\{1,5\}, \{5,9\}, \{3,8\}, \{4,8\}, \{2,7\}$ and $\{2,9\}$, all have the same sign. 
	
	The fact that the circuit is positive in all relevant edges implies that the map sending 
	$\OvAss{3}{9}$ to $\R^{k(n-2k)}=\R^6$  (mapping each vertex to the corresponding row vector of $\PP(t_1,\dots,t_9)$ and extending linearly in each face) contains the origin in its image. Hence, it does not produce a well-defined map $\phi_{\OvAss{3}{9},\VV}: |\OvAss{3}{9}|\to S^5$. Moreover, if we choose positions $\p'$ and $\p''$ close to $\p$ that are sufficiently generic to make $T$ a basis but with opposite orientations, then the degrees of the maps $|\OvAss{3}{9}|\to S^5$ obtained will differ by one unit.
\end{example}

\subsection{Cofactor rigidity does not realize $\OvAss{k}{n}$, for $n\ge 2k+6$, $k\ge 3$}
\label{sec:obstruction}

Although our main result in this section deals with the case $n\ge 2k+6$, for most of the section we assume $n=2k+3$ and characterize exacty when does cofactor rigidity $C_{2k}$ realize $\OvAss{k}{2k+3}$ as a complete fan. (We already saw in Theorem~\ref{thm:Desargues} that it not always does).

With $n=2k+3$ there are exactly $2k+3$ relevant edges, namely those of the  form $\{a,a+k+1\}$, for $a\in[n]$. These edges form a $(k+1)$-star $S$ that we call the \emph{relevant star}. We will normally consider the relevant edges in their ``star order'': the cyclic order in which $\{a,a+k+1\}$ is placed right after $\{a-k-1,a\}$.  Put differently, the edges of the relevant star, in their star order, are
\[
S= \{\{1,k+2\},\{1,k+3\},\{2,k+3\},\{2,k+4\},\ldots,\{k+1,2k+3\},\{k+2,2k+3\}\}
\]
Removing a number $\ell$ of edges of the relevant star results in $\ell$ paths counting as a ``path of length zero'' the empty path between two consecutive edges removed.

A simple counting shows that $k$-triangulations with $2k+3$ vertices are of the form $K_{2k+3}\setminus\{$3 edges$\}$. However, $K_{2k+3}$ minus three relevant edges is not always a $k$-triangulation. The necessary and sufficient condition is that the three paths obtained removing these edges are even. This ``evenness criterion'' is the reason why $\OvAss{k}{2k+3}$ is combinatorially a cyclic polytope of dimension $2k$ in $2k+3$ vertices.

In a similar way, the union of two adjacent $k$-triangulations is obtained removing two relevant edges from $K_{2k+3}$. We call such unions \emph{circuits} since we want to realize them as circuits in the vector configuration. The two relevant paths in a circuit $C$ necessarily have different parity, and the $k$-triangulations contained in $C$ are  obtained removing an edge that splits the odd path into two even paths. (For this to be doable in more than one way the odd path in $C$ must be of length at least three. But if $C$ equals $K_{2k+3}$ minus two edges and the odd relevant path in $C$ has length one then $C$ contains a $(k+1)$-crossing, so we are not interested in it).

Any codimension-two face $F$ of $\OvAss{k}{2k+3}$), in turn, is of the form ``$K_{2k+3}$ minus five relevant edges''. Let $\{a,b,c,d,e\}$ be the edges in star order, so that the relevant star minus $\{a,b,c,d,e\}$ consists of five paths (some of which may have length zero, as remarked above). The length of the elementary cycle of $F$ depends on the parity of the five paths, as follows (see Example \ref{ex:flips} for the first three cases with $k=2$):
\begin{itemize}
	\item If  the five paths are even, then the cycle has length five and the vertices of the cycle (that is, the  $k$-triangulations containing $F$) are obtained adding $\{a,b\}$, $\{b,c\}$, $\{c,d\}$, $\{d,e\}$, and $\{e,a\}$.
	\item If two consecutive paths,  say $(a,b)$ and $(b,c)$, are odd, then the cycle has length three and its vertices are formed with $\{a,b\}$, $\{b,c\}$ and $\{c,a\}$.
	\item If two non-consecutive paths,  say $(a,b)$ and $(c,d)$, are odd, then the cycle has length four and the vertices are formed with $\{a,c\}$, $\{c,b\}$, $\{b,d\}$ and $\{d,a\}$.
	\item If only one path is even then no multitriangulation contains $F$. $F$ is not really a face.
\end{itemize}
Hence, the only case with length 5 is the one with all intervals even.

We call a $k$ triangulation of the $(2k+3)$-gon \emph{octahedral} if its three missing edges have six distinct endpoints; equivalently, if the three relevant paths in it have positive length. (Observe that this needs $2k+3\ge 9$, hence $k\ge 3$). 
The reason for this name is that any such $k$-triangulation is, as a graph, an iterated cone (an odd number of times, greater than 1) over the graph of an octahedron.

\begin{lemma}\label{lemma:des39}
Consider a configuration $\q$ in convex position.
	Let $1\le i_1<i_2<i_3 <i_4\le 2k+3$ be such that $C:=K_{2k+3}\setminus\{\{i_1,i_3\}, \{i_2,i_4\}\}$ is a circuit and let $\lambda$ be its unique  (modulo a scalar factor) dependence in the cofactor matric $C_{2k}(\q)$; in particular, we assume that $i_3-i_1, i_4-i_2 \in \{k+1,k+2\}$. Let $\{i_1',i_3'\}$ be the first edge in the odd path of $S\setminus \{\{i_1,i_3\}, \{i_2,i_4\}\}$ (which can be $\{i_1\pm 1,i_3\}$ or $\{i_1,i_3\pm 1\}$) and let $\{j_1,j_2\}$ be another edge in the same path at a distance $d$ from the first. Then, we have that
	\[
	\sign(\lambda_{j_1j_2})=(-1)^d\sign(\lambda_{i_1'i_3'})
	\]
	if and only if the triangle formed by the three edges $\{i_1,i_3\}$, $\{i_2,i_4\}$ and $\{j_1,j_2\}$ is in the inner side of $\{i_1,i_3\}$ and $\{i_2,i_4\}$ (the side of length $k+2$).
	
	In particular, if all edges with $d$ even satisfy the condition, then the condition ICoP is satisfied by all flips contained in $C$.
\end{lemma}
\begin{proof}
	Without loss of generality we can suppose that $j_1<i_1<i_2<j_2<i_3=i_1+k+1<i_4=i_2+k+2$. Then $i_1'=i_1-1$, $i_3'=i_3$, and by the definition of star order, $j_1+j_2= i_1'+i_3'-d=i_1+i_3-d-1$.
	
	In the circuit, the degree of $i_3$ is $2k+1$ and by Lemma \ref{lemma:chsign},
	\[\sign(\lambda_{i_1'i_3'})=\sign(\lambda_{i_1-1,i_3})=(-1)^{i_1-j_1-1}\sign(\lambda_{j_1i_3})\]
	so the condition to be checked reduces to
	\[\sign(\lambda_{j_1j_2}\lambda_{j_1i_3})=(-1)^{d+1+i_1-j_1}=(-1)^{i_3-j_2}\]
	
	The circuit is obtained by a repeated coning from the $K_6$ without two diagonals, so that the original six vertices become $\{j_1,i_1,i_2,j_2,i_3,i_4\}$. The sign of an edge is inverted whenever we make a cone with the new vertex between the endpoints of that edge. As a result, the sign of $\lambda_{j_1j_2}\lambda_{j_1i_3}$ is inverted $i_3-j_2-1$ times exactly, so in the graph at the beginning, we should have $\lambda_{14}\lambda_{15}<0$.
	
	By Theorem \ref{thm:cofact}, this happens when the triangle formed by 14, 25 and 36 is negatively oriented. After making the cones, the configuration $\{j_1,i_1,i_2,j_2,i_3,i_4\}$ is negatively oriented and the triangle is in the side between $i_3$ and $i_4$, which is the inner side of the two edges.
\end{proof}

Observe that a relevant edge with $n=2k+3$ leaves $k$ points of the configuration on one side and $k+1$ on the other side. We call the one with $k+1$ points \emph{the big half-plane} defined by the relevant edge.

\begin{theorem}\label{thm:star}
Let $\q=(q_1,q_2 ,\dots ,q_{2k+3})$ be a configuration in convex position in $\R^2$. The following are equivalent:
\begin{enumerate}
	\item	$C_{2k}(\q)$ realizes $\OvAss{k}{2k+3}$ as a complete fan.
	\item For every octahedral $k$-triangulation $T$ the big half-planes defined by the three edges not in $T$ have non-empty intersection.
	\item The relevant star has ``non-empty interior'' (that is, the big half-planes of all relevant edges have non-empty intersection).
\end{enumerate}
\end{theorem}

\begin{remark}
It is interesting to note that, when condition (3) holds, any point $o$ taken in the ``interior'' of the relevant star makes the vector configuration $\{q_1-o,\dots q_{2k+3}-o\}$ be a Gale transform of the cyclic $2k$-polytope with $2k+3$ vertices. That is to say, the theorem says that $\q$ realizes $\OvAss{k}{2k+3}$ as a fan if and only if there is a point $o\in \R^2$ such that $\q-o$ is the Gale transform of a cyclic polytope. It seems to be a coincidence that the cyclic polytope in question is in fact isomorphic to $\OvAss{k}{2k+3}$.
\end{remark}

\begin{proof}
	The implication (3)$\Rightarrow$(2) is trivial. Let us see the converse.
	First observe that, by Helly's Theorem, the intersection of all half-planes is non-empty if, and only if, the intersection of every three of them is non-empty. So, we only need to show that, when condition (3) is restricted to three half-planes, only the case where the half-planes come from the missing edges in an  octahedral triangulation matters.

	So, consider three relevant  edges $\{\{i_1,i_4\},\{i_2,i_5\},\{i_3,i_6\}\}$ and their corresponding big half-planes. We look at the three paths obtained in the relevant star when removing these three edges.
	If at least one path (and hence exactly two) has odd length, then the intersection of the three big half-planes is automatically non-empty: let the even interval be $(i_6, i_1)$. Then the edge $\{i_2,i_5\}$ crosses the other two and leaves both $i_1$ and $i_6$ in its big  half-plane, so we can always find a point in the intersection of the three half-planes in a neighborhood of the intersection of the lines containing $\{i_1,i_4\}$ and $\{i_3,i_6\}$.
	
	Similarly, if two of the three edges are consecutive (say $i_6=i_1$), then the intersection of their two half-planes is an angle of the relevant star. This angle necessarily meets both of the half-planes defined by the third edge $\{i_2,i_5\}$, so the intersection is again non-empty.
	
	That is, the only case of three edges whose big half-planes might perhaps produce an empty intersection is when the three paths they produce are even and non-empty. This is exactly the same as saying that they are the three missing edges of an octahedral $k$-triangulation, which proves (2)$\Rightarrow$(3).
	
	Now, the implication (1)$\Rightarrow$(2) follows from the previous lemma: if the complete fan is realized, the condition ICoP is satisfied in the flips from an octahedral triangulation, in particular, the flip from $K_{2k+3}\setminus\{\{i_1,i_3\},\{i_2,i_4\},\{j_1,j_2\}\}$ that removes $\{i_1',i_3'\}$ and inserts $\{j_1,j_2\}$, with $d$ even. By Lemma \ref{lemma:des39}, this is equivalent to saying that the big half-planes of $\{i_1,i_3\}$, $\{i_2,i_4\}$ and $\{j_1,j_2\}$ intersect, which covers all the cases in (2). So it only remains to show that (2)$\Rightarrow$(1).
	
	If (2), or equivalently (3), holds then we know, by the previous argument, that flips from an octahedral triangulation satisfy ICoP. These are exactly the flips whose two missing edges do not share a vertex. The other flips must be of the form $K_{2k+3}$ minus two edges with a common end-point.
	Hence, they contain a $K_{2k+2}$, in which the signs are as predicted by Lemma~\ref{lemma:chsign} and the ICoP property also holds in them.
	
	To finish the proof, we just need to check the condition about elementary cycles of length 5. Given one of these cycles, adding three consecutive edges of the five in the cycle gives the graph of a flip. If two edges in the cycle share a vertex, we can add the other three edges to get the graph of a flip that contains a $K_{2k+2}$, which has the two flipping edges as diameters and the other edge with length $k$, so it has opposite sign. Otherwise, the five edges are disjoint.
	
	In this case, let $\{a,b,c,d,e\}$ be the edges. By the condition (3), their five big half-planes have non-empty intersection. Without loss of generality, suppose that $b$ is a side of that intersection. Adding $a$, $b$ and $c$, we get the graph of a flip. As $b$ is a side of the intersection, the triangle formed by $b$, $d$ and $e$ is inside the big half-planes of $d$ and $e$, so we can again apply Lemma \ref{lemma:des39} to get that $b$ has opposite sign to $a$ and $c$, as we wanted to prove.
\end{proof}

\begin{corollary}\label{coro:2k+3}
For every $k$ there are point configurations $\q$ such that $C_{2k}(\q)$ realizes $\OvAss{k}{2k+3}$ as a fan. For example, the vertices of a regular $(2k+3)$-gon.
\end{corollary}

\begin{proof}
	The barycenter of the $2k+3$-gon lies in the interior of the relevant star.
\end{proof}

This theorem also implies Theorem \ref{thm:2k+6}. Cofactor rigidity with points in convex position cannot realize $\OvAss{k}{n}$ as a fan for $n\ge 2k+6$ and $k\ge 3$:


\begin{proof}[Proof of Theorem \ref{thm:2k+6}]
Let $\q=(q_1,\dots, q_n)$ be a configuration in convex position.
By Lemma \ref{lemma:monotone} we only need to show the case $n=2k+6$. 

Let $I_1 = [n] \setminus \{4,{k+5},{k+9}\}$ and $I_2 = [n] \setminus \{2,6,{k+7}\}$. Then $\q|_{I_1}$ and $\q|_{I_2}$ are configurations with $2k+3$ points, to which we can apply Theorem \ref{thm:star}. We consider their respective $k$-triangulations
$T_1=K_{I_1}\setminus \{\{1,k+4\}, \{3,k+6\}, \{5,k+8\}\}$ and 
$T_2=K_{I_2}\setminus \{\{1,k+4\}, \{3,k+6\}, \{5,k+8\}\}$. This theorem tells us that in order for $\q_{I_1}$  to realize $\OvAss{k}{2k+3}$ we need 
$(q_1, q_3,q_5, q_{k+4}, q_{k+6}, q_{k+8})$ to be negatively oriented, and in order for $\q_{I_2}$  to realize it  we need the same configuration to be positively oriented.
This is a contradiction, so one of the two does not realize $\OvAss{k}{2k+3}$. Lemma \ref{lemma:monotone} implies that $\q$ does not realize $\OvAss{k}{2k+6}$.
\end{proof}

\section{Positive results on realizability, for $k=2$}
\label{sec:triangsrigid}

In this section, we prove Conjecture \ref{conj:rigid} for $k=2$ and Theorem \ref{thm:fan-k=2}.

\subsection{$2$-triangulations are bases in $\PP_2(n)$ }
\label{subsec:vertex-split}

To  prove Theorem~\ref{thm:rigid} (that is, Conjecture \ref{conj:rigid} for $k=2$) we need operations that send a 2-triangulation in $n$ vertices to one in $n+1$ vertices, and viceversa. These operations are called the \emph{inflation of a $2$-crossing}, and  the \emph{flattening of a star}. They are defined for arbitrary $k$ in \cite[Section 7]{PilSan}, 
but we use only the case $k=2$.

An \emph{external $2$-crossing} is a $2$-crossing with two of the end-points consecutive. An \emph{external $2$-star} is one with three consecutive points. Equivalently, one using a boundary edge (an edge of length two).
Let $u,v,w$ be three consecutive vertices of the $(n+1)$-gon, and consider the $n$-gon obtained by removing the middle vertex $v$. The identity map on vertices then induces a bijection between external $2$-stars in the $(n+1)$-gon using the boundary edge $\{u,w\}$ and external $2$-crossings in the $n$-gon using the vertices $u, w$. If, to simplify notation, we let $[n+1]$ and $[n]$ be our vertex sets, with $u=n$, $v=n+1$ and $w=1$, the bijection is
\[
S =\{\{n,1\},\{1,c\},\{c,n+1\},\{n+1,b\},\{b,n\}\} \quad\leftrightarrow\quad C =\{\{n,b\},\{1,c\}\}.
\]

From now on let us fix an external star $S\subset \binom{[n+1]}{2}$ and its corresponding external crossing $C\in \bnn$, of the above form. Consider the set $\link_2(S)^0$ (respectively $\link_2(C)^0$) of relevant edges that do not form a $3$-crossing with $S$, (respectively, with $C$); they are, respectively, the sets of vertices in $\link_{\OvAss{2}{n+1}} (S)$ and in $\link_{\OvAss{2}{n}} (C)$.

\begin{theorem}[{\cite[Section 7]{PilSan}}]
\label{thm:flattening}
The following map is a bijection
\[
\begin{array}{crcl}
\phi:& \link_2(S)^0 &\to &\link_2(C)^0 \\
& \{i,j\} & \mapsto & \{i,j\}  \quad \text{if $n+1\not \in \{i,j\}$ }  \\
& \{i,n+1\} & \mapsto & 
  \begin{cases}
      \{i,n\} &  \text{if $ 1 \le i < b$ } \\
      \{1,i\} &  \text{if $ c < i \le n$ } \\
  \end{cases}    
\end{array}
\]
and it induces an isomorphism of simplicial complexes
\[
\hat \phi:\link_{\OvAss{2}{n+1}} (S) \stackrel{\cong}{\longrightarrow} \link_{\OvAss{2}{n}} (C).
\]
\end{theorem}

\begin{proof}
The map is well defined because $\{i,n+1\}$ is in $\link_2(S)^0$ if and only if $i\not\in[b,c]$.

It is injective because the only edges that could have the same image are $\{i,n\}$ and $\{i,n+1\}$ if
$1\le i < b$ or $\{1,j\}$ and $\{j,n+1\}$ if $c<j\le n$. But in the first case $\{i,n\}$ would form a $3$-crossing with $\{b,n+1\}$ and $\{c,1\}$, and in the second case $\{1,j\}$ would form a $3$-crossing with $\{b,n\}$ and $\{c,n+1\}$.

It is surjective because if $\{i,j\}\in \bnn$ is not in the image of $\phi$ then the first case in the definition implies $\{i,j\}\not \in \link_2(S)^0$, but the only edges in $\bnn \setminus \link_2(S)^0$ are those with $1\le i<b$ and $c\le j<n$. Among these, the only ones in $\link_2(C)^0$ are those with $i=b$ or $j=c$, which are in the image of $\phi$.

To show that it induces an isomorphism of the complexes we need to check that if $T\subset \binom{[n+1]}{2}$ is $3$-crossing-free and contains $S$ then $\phi(T\setminus S)\cup C$ is also $3$-crossing-free, and vice-versa. These are essentially Lemmas 7.3 and 7.6 in \cite{PilSan}.
\end{proof}

\begin{definition} [Flattening of a star, inflation of a crossing]
\label{defi:inflation}
\label{defi:dilation}
Let $e\in\binom{[n+1]}{2}$ be a boundary edge, let $S\subset\binom{[n+1]}{2}$ be an external $2$-star using $e$, and let $T$ be a $2$-triangulation containing $S$. 

Let $\hat\phi$ be the isomorphism of Theorem \ref{thm:flattening} (after a cyclic relabelling sending $e$ to $\{n,1\}$).
We say that the $2$-triangulation $\hat\phi(T)$ is the \textit{flattening of $e$ in $T$}, and denote it $\underline{T}_e$.
We also say that $T$ is the \textit{inflation} of $C$ in $\hat\phi(T)$.
%
%
\end{definition}

The crucial fact  that we need is, under certain conditions, \emph{inflation of a $2$-crossing} is a particular case of a \emph{vertex split}.

\begin{definition}[Vertex $d$-split]
	A \emph{vertex $d$-split} in a graph $G=(V,E)$ consists in changing a vertex $u\in V$, with degree at least $d-1$, into two vertices $u_1$ and $u_2$ joined by an edge and joining all neighbours of $u$ to at least one of $u_1$ or $u_2$, and exactly $d-1$ of the neighbors to both.
\end{definition}

Put differently, the graph $G'$ with vertex set $V\setminus \{u\} \cup \{u_1,u_2\}$ is a vertex $d$-split of a graph $G$ on $V$ if, and only if: $G'$ contains the edge $u_1u_2$, the contraction of that edge produces $G$, and $u_1$ and $u_2$ have exactly $d-1$ common neighbors in $G'$.

\begin{lemma}[\protect{\cite[proof of Theorem 8.7]{PilSan}}]
\label{lemma:inflation}
Inflation of a ``doubly external'' $2$-crossing of the form $C=\{\{n,b\}, \{1,n-1\}\}$ in a $2$-triangulation $T$ is an example of vertex $4$-split at $n$, with new vertices $n$ and $n+1$. The inflated star has four consecutive vertices $n-1,n,n+1$ and $1$.
\end{lemma}

\begin{proof}
Let $T'$ be the inflated $2$-triangulation.
Plugging $c=n-1$ in Theorem \ref{thm:flattening} gives
\[
\begin{array}{crcl}
\phi:& \link_2(S)^0 &\to &\link_2(C)^0 \\
& \{i,j\} & \mapsto & \{i,j\}  \quad \text{if $n+1\not \in \{i,j\}$ }  \\
& \{i,n+1\} & \mapsto &\{i,n\}   \quad \text{if $ 1 \le i < b$ } \\
& \{n,n+1\} & \mapsto &\{1,n\},
\end{array}
\]
which implies that the relevant edges of $T$ are indeed obtained from those of $T'$ by identifying the vertices $n$ and $n+1$. The same happens for the irrelevant and boundary edges (which are independent of $T$ and $T'$).
It also implies that all the neighbors of $n$ in $T$ are neighbors of at least one of $n$ and $n+1$ in $T'$.

Hence, we only need to check that $n$ and $n+1$ have exactly three common neighbors in $T'$. This holds since $n-1$, $1$ and $b$ are common neighbors of $n$ and $n+1$ in $T'$, and any additional common neighbour $b'$ would create a $3$-crossing with $\{n-1,1\}$ and either $\{n,b\}$ if $b'>b$ or $\{n+1,b\}$ if $b'<b$.
\end{proof}

That vertex $d$-splits preserve independence in both $\RR_d(n)$ and $\CC_d(n)$ is a classical result~\cite[pp.~68 and Remark 11.3.16]{Whiteley}. Preserving independence holds also in $\HH_d(n)$ and $\PP_d(n)$:

\begin{proposition}[\protect{\cite[Proposition 4.10]{CreSan:moment}}] \label{prop:ext}
\label{prop:split}
Corank does not increase under vertex $d$-split neither in $\HH_d$ nor in $\PP_d$. Hence,
vertex $d$-splits of independent graphs are independent, both in $\HH_d$ and $\PP_d$.
\end{proposition}

This has the following consequence. We only state and prove it for $\HH_d(\q)$ (which includes the case of $\PP_d(\bt)$
when points are chosen along the moment curve) but the same statement, with the same proof, holds also for $\RR_d(\q)$ and $\CC_d(\q)$.

\begin{corollary}
\label{coro:split}
Let  $\q=(q_1,q_2 ,\dots ,q_{n})$ be a configuration in $\R^d$ (resp. in the moment curve) and assume that a certain graph $G$ is a circuit in  $\HH_d(\q)$. Let $G'$ be a vertex $d$-split of $G$ and consider it embedded in positions $\q'$ that are generic (resp. generic along the moment curve) and sufficiently close to $\q$. 

Then, $G'$ is either independent in  $\HH_d(\q')$ or it contains a unique circuit. If the latter happens, then the signs of the non-splitting edges are preserved.
\end{corollary}

\begin{proof}
First perturb $\q$ to be generic, which either makes it independent or maintains it being a circuit. Proposition~\ref{prop:split} implies that $G'$ is independent  in the first case and that it is either independent or it contains a unique circuit in the second case. 

So, we only need to show that if the latter happens then all the non-spliting edges are part of the circuit and that they preserve their signs. That they are part of the circuit follows again from the proposition: if $e\in G'$ is not a spliting edge then it comes from an edge of $G$. Now, since $G$ was a circuit, $G\setminus e$ was independent; hence $G'\setminus e$ is also independent, so $e$ belongs to the circuit.

To see that the signs are preserved, consider moving the points continuously from $\q'$ back to $\q$. (At the end, the two vertices created in the split collide into the same position but we can still consider them two different vertices of the graph $G'$, with a degenerate embedding). 
Since the positions of $\q'$ are taken sufficiently close to those of $\q$, this continuous motion can be made through positions at which $G'$ always has a unique dependence, and such that the signs of the dependence do not change except perhaps at the end of the path, when we get to $\q$. At the end of the path the dependence must degenerate to the original (unique) dependence of $G$, in the sense that the coefficients of non-splitting edges are the same in $G$ and $G'$, and the coefficients of the splitting edges in $G'$ add up to those in $G$.
Now, since the signs of the non-splitting edges are never zero along the path and still non-zero at the end, by continuity they must be preserved.
\end{proof}

 \emph{Dependence}, however, is not preserved. See example after \cite[Proposition 4.10]{CreSan:moment}.

\begin{definition}[Ears]
	A star in a 2-triangulation is \textit{doubly external} if it has four consecutive vertices, like the ones that can be obtained in Lemma \ref{lemma:inflation}.
	
	An \textit{ear} of the 2-triangulation is an edge of length 3.
\end{definition}

For every ear $\{a,a+3\}$ in a $2$-triangulation $T$ there is a unique star in $T$ using the vertices $a$, $a+1$, $a+2$ and $a+3$ (hence, a doubly external star), and it has $\{a,a+2\}$, $\{a,a+3\}$ and  $\{a+1,a+3\}$ as  three consecutive sides. We call this star the star \textit{bounded by the edge $\{a,a+3\}$}.

\begin{theorem}[{\cite[Corollary 6.2]{PilSan}}]
	\label{thm:ears}
	The number of ears in a 2-triangulation is exactly 4 more than the number of internal stars.
\end{theorem}

\begin{proof}[Proof of Theorem~\ref{thm:rigid}]
	This is proved by induction in $n$. For $n=5$, the only 2-triangulation is $K_5$, that is a basis in 4 dimensions.

	Suppose the statement is true for $n$ and take a 2-triangulation $T'$ on $n+1$ vertices. By Theorem \ref{thm:ears}, $T'$ has at least four ears. Without loss of generality, suppose one of them is $\{n-1,1\}$. Then it bounds a doubly external star with the vertices $n-1,n,n+1$ and 1. Let $b$ be the remaining vertex in the star. Then, $T'$ can be obtained from a 2-triangulation $T$ in $n$ vertices inflating the 2-crossing $\{\{n,b\},\{1,n-1\}\}$. By Lemma \ref{lemma:inflation}, this operation is a vertex $4$-split, that preserves independence by Proposition \ref{prop:split}.
\end{proof}

\subsection{Realizing individual elementary cycles}
\label{subsec:cycles}

We now prove some results for realizability as a fan in the case $k=2$. Among other things, we show that for $n\le 7$ any position of the points in the plane will realize the multiassociahedron $\OvAss2{n}$ as a fan via the cofactor rigidity matrix.

\begin{corollary}
\label{coro:2k+2}
For $n=2k+2$, any choice of $q_1,\dots, q_{2k+2}\in\R^2$ in convex position makes the rows of the cofactor matrix realize $\OvAss{k}{n}$ as a polytopal fan.
\end{corollary}

\begin{proof}
In this case, all the $k$-triangulations are $K_{2k+2}$ minus a diameter. The simplicial complex is in this case the boundary of a $(k-1)$-simplex. We also know that $K_{2k+2}$ is a circuit, therefore all the $k$-triangulations are bases.

In this circuit, Lemma~\ref{lemma:chsign} implies that the sign of each edge coincides with the parity of its length. The flipped edges in this case are two of the diameters, that have all the same sign. Hence, the condition ICoP in Corollary \ref{coro:fan} is true. The condition on the elementary cycles is trivial because all of them have length three.

This implies that every position of the points realizes the boundary of the simplex as a complete fan. But realizing a simplex as a complete fan is equivalent to realizing it as a polytope, so the corollary is proved.
\end{proof}

\begin{corollary}
\label{coro:small-n}
For $k=2$ and $n=7$, any choice of $q_1,\dots, q_7\in\R^2$ in convex position realizes $\OvAss{2}{7}$ as a fan.
\end{corollary}

\begin{proof}
By Theorem \ref{thm:star}, the fact that a position realizes the fan is equivalent to the interior of the 3-star formed by the seven points being non-empty. This is trivial, because any three edges without common vertices are consecutive in the circle, and the seventh vertex will always be in the intersection.
\end{proof}

We now look at the case $k=2$ and $n\ge 8$. We are going to show that for each elementary cycle there are positions that make that cycle simple (and, in particular, for every flip there is an embedding that makes ICoP hold for that flip). Of course, this does not imply that $\OvAss{2}{n}$ can be realized as a fan; for that we would need fixed positions that work for all cycles, not one position for each cycle. But this implication would hold if $2$-triangulations were basis at arbitrary positions (Theorem \ref{thm:fan-k=2}).

We  need the fact that, in this case, neither a flip nor an elementary cycle of length 5 can use all the doubly external stars in a $2$-triangulation.

\begin{lemma} Let $T$ be a 2-triangulation on at least 8 vertices.
\begin{enumerate}
\item For any relevant edge $e$ in $T$ there is a doubly external star in $T\setminus \{e\}$.
\item If $e$ and $f$ are two relevant edges in  $T$ and the elementary cycle $\link_{\OvAss{2}{n}}(T\setminus \{e,f\})$ has length five then
there is a doubly external star in $T\setminus \{e,f\}$.
\end{enumerate}
\end{lemma}

\begin{proof}
A star cannot be bounded by more than two ears, and if a star $S$ is bounded by two ears then its five vertices are consecutive. That is, the edges of $S$ are the two ears plus three boundary edges. We call such stars \emph{triply external}. 

Two triply external stars cannot have a common edge (except for $n=6$, but we are assuming $n\ge 8$).  This, together with the fact that $T$ has at least four ears (Theorem~\ref{thm:ears})  implies part (1): if $T$ has two triply external stars then (at least) one of them does not use $e$, and if $T$ has one triply external star, or none, then the existence of four ears implies that there are at least three doubly external stars in $T$, and only two of them can use $e$.

	We now look at part (2) of the statement. By the proof of Corollary~\ref{coro:cycles}, for the link of $T\setminus \{e,f\}$ to have length five we need that there is a star $S_0$ in $T$ using both $e$ and $f$, plus another two stars $S_e$ and $S_f$ using each one of $e$ and $f$. We want to show that there is a doubly external star that is not any of these three. (These three may or may not be doubly external, or external).
	
	Suppose, to seek a contradiction, that no star other than these three is doubly external. Then every ear bounds one of these three stars. Since only triply external stars are bounded by two (and then only two) ears, the total number of ears is at most three plus the number of triply external stars among $S_0$, $S_e$ and $S_f$. The three cannot be triply external (because $S_0$ shares edges with both $S_e$ and $S_f$), so the number of ears is at most five. But five ears would imply $S_e$ and $S_f$ to be triply external, in particular $e$ and $f$ to be ears bounding them, and $S_0$ to be doubly external, bounded by the fifth ear, different form $e$ and $f$. In particular, the five edges of $S_0$ would have to be two boundary edges and three ears, which can only happen with $n=7$ (because two of the ears would need to share a vertex and have their other end-points consecutive).

	So, there are at most four ears in $T$ and, by Theorem~\ref{thm:ears}, exactly four. Moreover, they all bound $S_0$, $S_e$ or $S_f$. 
	By Theorem \ref{thm:ears} again, this implies that each of the other $n-7$ stars in $T$  is an external star, but not a doubly external one. That is, each of these $n-7$ stars contains one and only one of the $n-7$ boundary edges, and the other seven boundary edges are distributed among $S_0$, $S_e$ and $S_f$. 

Let $T_0$ be the $2$-triangulation of the $7$-gon obtained by flattening one by one the $n-7$ boundary edges not in $S_0$, $S_e$ or $S_f$. Observe that, when we flatten a singly external star, all other stars have the same number of boundary edges before and after the flattening. In particular, the last star that was flattened was still singly external before the flattening, so it becomes a singly external $2$-crossing (that is, a crossing of two relevant edges) in $T_0$.
At the end, in $T_0$ only $S_0$, $S_e$ and $S_f$ survive, and the edges $e$ and $f$ are such that their link is a cycle of length five (because all throughout the process the link of $T\setminus \{e,f\}$ preserves its length, by Theorem \ref{thm:flattening}). 

The contradiction is that for the cycle to be of length five we need the two relevant edges in $T\setminus \{e,f\}$ to be non-crossing, as in the third picture of Example~\ref{ex:flips}, but those two edges must cross because they are the $2$-crossing obtained form the last star that was flattened.
\end{proof}

\begin{theorem}
\label{thm:good_cycle}
\begin{enumerate}
	\item For each pair of adjacent facets in $\OvAss{2}{n}$ there is a choice of parameters for $P_4(t_1,\dots,t_n)$ that makes the corresponding circuit of the polynomial rigidity matrix satisfy ICoP.
		\item For each elementary cycle with length 5 in $\OvAss{2}{n}$ there is a choice of parameters for $P_4(t_1,\dots,t_n)$ that makes the cycle simple.
\end{enumerate}
\end{theorem}

\begin{proof}
	The proof goes by induction in $n$. For $n\le 7$ it is already proved in corollaries \ref{coro:2k+2} and \ref{coro:small-n}, so suppose it is true for $n$ and prove it for $n+1\ge 8$.
	
	For the first part, let $T_1$ and $T_2$ be two $2$-triangulations we are looking at, and let $e,f$ be the edges in $T_1\setminus T_2$ and $T_2\setminus T_1$, respectively. Part (1) of the previous lemma implies that there is a doubly external star $S\subset T_1\setminus \{e\} = T_1\cap T_2$.	
	By Theorem \ref{thm:flattening}, flattening $S$ in $T_1$ and $T_2$ we get $2$-triangulations $T_1'$ and $T_2'$ that still differ by a flip, and by the inductive hypothesis the sign condition ICoP will hold in the circuit $T_1'\cup T_2'$ for certain choice of the parameters $t_i$. Now we return to the flip graph in $n+1$ vertices by the reverse operation of flattening a doubly external star, which is a vertex split by Lemma~\ref{lemma:inflation}.
If we keep the two split vertices close enough, the signs of the non-split edges (which include the flip edges $e$ and 
$f$) will not be altered (Corollary~\ref{coro:split}), and the ICoP condition still holds.
	
	For the second part, let our elementary cycle be (the link of) $T\setminus \{e,f\}$, for a $2$-triangulation $T$ and relevant edges $e,f\in T$. By the previous lemma, there is a doubly external star $S\subset T\setminus \{e,f\}$.
	
	Again, we can flatten $S$, assume by the inductive hypothesis that the sign condition in part (3) of Corollary~\ref{coro:fan} holds in the flattened $5$-cycle, and return to $n+1$ vertices by a vertex split that will preserve the sign condition if the split vertices are kept close enough.
\end{proof}

\begin{proof}[Proof of Theorem \ref{thm:fan-k=2}]
	Suppose, for a contradiction, that all positions realize $\OvAss{2}{n}$ as a basis collection, but there is a position $\bt$ that does not realize it as a fan. This implies that, at $\bt$, there is an elementary cycle with wrong signs. But, by Theorem \ref{thm:good_cycle}, there is another position $\bt'$ giving the right signs in that cycle.
	
	Consider now a continuous transition between $\bt$ and $\bt'$. At some point, the signs need to change, either by attaining condition ICoP at a flip, or by making the cycle simple. But any of the two ways would involve collapsing some cone to lower dimension at that point, which does not happen by hypothesis.
\end{proof}

\subsection{Experimental results}
\label{subsec:experiments}
In this section we report on some experimental results. In all of them we choose real parameters $\bt= \{t_1< t_2<\dots<t_n\}$ (actually we choose them integer, so that they are exact) and computationally check whether the configuration of rows of $P_{2k}(t_1,\dots,t_n)$ realizes $\OvAss{k}{n}$ first as a collection of bases, then as a complete fan, and finally as the normal fan of a polytope. 

For the experiments we have written python code which, with input $k$, $n$ and the parameters $\bt$, first constructs the set of all $k$-triangulations and then checks the three levels of realizability as follows:
\begin{enumerate}
\item Realizability as a collection of bases amounts to computing the rank of the submatrix $P_{2k}(\bt)|_T$ corresponding to each $k$-triangulation $T$.
\item For realizability as a fan we first check the ICoP  property, which amounts to computing the signs of certain dependences among rows  of $P_{2k}(\bt)$. There is une such dependence for each ridge in the complex, so the total number of them is $N D/2$ where $N$ is the number of $k$-triangulations on $n$ points and $D=k(n-2k-1)$ is the dimension of $\OvAss{k}{n}$.

If ICoP holds then we check that a certain vector lies in the positive span of a unique facet of the complex. We do this for the sum of rays corresponding of a particular $k$-triangulation, the so-called \emph{greedy} one. This property, once we have ICoP, is equivalent to being realized as a fan by parts (2) and (3) of Theorem~\ref{thm:winding}.%

The \emph{greedy} $k$-triangulation is the (unique) one containing all the irrelevant edges and the edges in the complete bipartite graph $[1,k]\times [k+1,n]$, and only those. It is obvious that these edges do not contain any $(k+1)$-crossing and we leave it to the reader to verify that the number of relevant ones is indeed $k(n-2k-1)$.

\item For realizability as a polytope we then need to check feasibility of the linear system of inequalities (\ref{ineq}) from Lemma \ref{lemma:inequalities}.

Here, without loss of generality we can assume that the lifting vector $f_{ij}$ is zero in all edges of a particular $k$-triangulation, and we again use the greedy one. This reduces the number of variables in the feasibility problem from $n(n-2k-1)/2$ to $(n-2k)(n-2k-1)/2$, a very significant reduction for the values of $(n,k)$ where we can computationally construct $\OvAss{k}{n}$. Apart of the computational advantage, it saves space when displaying a feasible solution; in all the tables in this section we show only the non-zero values of $f_{ij}$, which are those of relevant edges  contained in $[k+1,n]$. 
Note that taking all the $f_{ij}$'s of a particular $k$-triangu\-lation equal to zero makes the rest strictly positive.
\end{enumerate}

\begin{remark}
\label{rem:monotone}
If a choice of parameters realizes $\OvAss{k}{n}$ (at any of the three levels) for a certain pair $(k,n)$ then deleting any $j$ of the parameters the same choice realizes $\OvAss{k'}{n-j}$ for any $k'$ with $ k-j/2\le k'\le k$. This follows from Lemma~\ref{lemma:monotone} plus the fact that each of the three levels of realization is preserved by taking links.
\end{remark}

Our first experiment is taking equispaced parameters. Since an affine transformation in the space of parameters produces a linear transformation in the rows of $P_{2k}(\bt)$, we take without loss of generality  $\bt=(1,2,3,\dots,n)$. We call these the \emph{standard positions} along the parabola.

For $k\ge 3$ and $n\ge 2k+3$ we show in Theorem \ref{thm:Desargues} that standard positions do not even realize $\OvAss{k}{n}$ as a collection of bases. Hence, we only look at $k=2$.

\begin{lemma}
\label{lemma:standard}
Let $\bt=\{1,2,\dots,n\}$ be standard positions for the parameters. Then:
\begin{enumerate}
\item Standard positions for $P_{4}(\bt)$ realize $\OvAss{2}{n}$ as the normal fan of a polytope if and only if $n\le 9$.

\item The non-standard positions $\bt=(-2,1,2,3,4,5,6,7,9,20)$ for $P_{4}(\bt)$ realize $\OvAss{2}{10}$ as the normal fan of a polytope.

\item Standard positions for $P_{4}(\bt)$ realize $\OvAss{2}{n}$ as a complete fan for all $n\le 13$.
\end{enumerate}
\end{lemma}

\begin{proof}
For part (1), by Lemma~\ref{lemma:standard} we only need to check that $n=9$ works and $n=10$ does not. For $n=8,9$ Table~\ref{table:2n} 
shows values of $(f_{ij})_{i,j}$ that prove the fan polytopal. For $n=10$ the computer said that the system is not feasible (which finishes the proof of part (1)), but modifying the standard positions to the ones in part (2) it gave the feasible solution displayed in Table~\ref{table:210}.

For part (3), the computer checked the conditions for a complete fan for $n=8,9,10,11,12,13$. Only the last one would really be needed; this last one took about 7 days of computing in a standard laptop.
\end{proof}

\begin{table}
	\begin{tabular}{c|c}
		 $i,j$ & $f_{ij}$ \\ \hline
		3,6 & 3 \\
		3,7 & 14 \\
		3,8 & 36 \\
		4,7 & 3 \\
		4,8 & 16 \\
		5,8 & 6 \\
	\end{tabular}
\qquad\qquad\qquad
	\begin{tabular}{c|c}
		 $i,j$ & $f_{ij}$ \\ \hline
		3,6 & 7 \\
		3,7 & 29 \\
		3,8 & 76 \\
		3,9 & 165 \\
		4,7 & 9 \\
	\end{tabular}\quad
	\begin{tabular}{c|c}
		 $i,j$ & $f_{ij}$ \\ \hline
		4,8 & 33 \\
		4,9 & 95 \\
		5,8 & 6 \\
		5,9 & 42 \\
		6,9 & 16 \\ 
	\end{tabular}
	\medskip
	\label{res9}
	\caption{Height vectors $(f_{ij})_{i,j}$ realizing $\OvAss{2}{8}$ (left) and $\OvAss{2}{9}$ (right) as a polytopal fan with rays in $P_4(1,2,\dots,n)$ (standard positions).}
	\label{table:2n}
\end{table}

\begin{table}
	\begin{tabular}{c|c}
		 $i,j$ & $f_{ij}$ \\ \hline
		3,6 & 44 \\
		3,7 & 161 \\
		3,8 & 424 \\
		3,9 & 1733 \\
		3,10 & 46398 \\
	\end{tabular}\quad
	\begin{tabular}{c|c}
		 $i,j$ & $f_{ij}$ \\ \hline
		4,7 & 45 \\
		4,8 & 260 \\
		4,9 & 1722 \\
		4,10 & 60296 \\
		5,8 & 106 \\
	\end{tabular}\quad
	\begin{tabular}{c|c}
		 $i,j$ & $f_{ij}$ \\ \hline
		5,9 & 1062 \\
		5,10 & 42019 \\
		6,9 & 196 \\
		6,10 & 13048 \\
		7,10 & 6146 \\ 
	\end{tabular}
	\medskip
	\caption{A lifting vector that leads to a polytopal realization of the multiassociahedron for $k=2$ and $n=10$, with $(t_i)_{i=1,...,10}=(-2,1,2,3,4,$ $5,6,7,9,20)$.}
	\label{res10}
	\label{table:210}
\end{table}


Let us mention that for $n=11$ we have tried several positions besides the standard ones. All of them realize the complete fan but none realize it as polytopal. Among the positions we tried are the ``equispaced positions along a circle'' that we now explain.

The standard parabola is projectively equivalent to the unit circle. Since $P_{2k}(\bt)$ is linearly equivalent to the cofactor rigidity matrix $C_{2k}(\q)$ for the points $\q$ of the parabola corresponding to the parameters $\bt$, it makes sense to look at the values of $\bt$ that produce equispaced points (that is, a regular $n$-gon) when the parabola is mapped to the circle. We call those values of $\bt$, ``equispaced along the circle''. They are
\begin{equation}
t_i = \tan\left(\alpha_0+ \frac{i}{n}\pi\right), \quad i=1,\dots,n
\label{eq:circle}
\end{equation}
for any choice of $\alpha_0$, with a symmetric choice being $\alpha_0=-(n+1)\pi/2n$.

For $k=3$, $n=10$, we already know that  standard positions do not even give a basis collection. We have tried to strategies to realize $\OvAss{3}{10}$: perturbing the standard positions slightly we have been able to recover independence (that is, a basis collection), but not the fan (the ICoP condition was not satisfied). 
Using equispaced positions along the circle via Equation~\eqref{eq:circle} the positions realize the polytope.

For $k=3$, $n=11$ and for $k=4$ and $n=12,13$ equispaced positions realize the fan but not the polytope, even after trying several perturbations. 

\begin{lemma}
\label{lemma:310}
\begin{enumerate}
\item For the same positions $\bt$ of Table~\ref{table:210}, $P_6(\bt)$ realizes $\OvAss{3}{10}$ as the normal fan of a polytope.
The following are valid values of $f$:
\begin{center}
	\begin{tabular}{c|c}
		 $i,j$ & $f_{i,j}$ \\ \hline
		4,8 & 4 \\
		4,9 & 69 \\
		4,10 & 16074 \\
		5,9 & 14 \\
		5,10 & 10281 \\
		6,10 & 3948 \\
	\end{tabular}
	\qquad\qquad
	\medskip
\end{center}

\item Equispaced positions along the circle realize $\OvAss{k}{n}$ as a fan for $(n,k)\in \{(3,11),(4,12),(4,13)\}$.
\qed
\end{enumerate}
\end{lemma}

\end{document}